\documentclass[10pt]{article}

\usepackage{amsmath,amsthm,amssymb,bbm,bm}
\usepackage[bookmarksnumbered=true,pagebackref]{hyperref}
\usepackage{tocbibind}
\usepackage{enumitem}
\usepackage{mleftright}
\usepackage{graphicx,caption,subcaption}

\newtheorem{theorem}{Theorem}
\newtheorem*{theorem*}{Theorem}

\newtheorem{proposition}{Proposition}
\newtheorem{lemma}{Lemma}

\theoremstyle{remark}
\newtheorem{remark}{Remark}
\newtheorem*{acknowledgments}{Acknowledgments}

\theoremstyle{definition}
\newtheorem{definition}{Definition}

\newtheoremstyle{notes}
{3pt}
{3pt}
{}
{}
{\bfseries}
{:}
{.4em}
{}
\theoremstyle{notes}
\newtheorem*{keywords}{Keywords}
\newtheorem*{subjclass}{AMS MSC 2020}

\newcommand{\D}[1]{\mathop{\mathrm{d}#1}}
\DeclareMathOperator{\conv}{CH}
\DeclareMathOperator{\hull}{TH}
\DeclareMathOperator{\strh}{SH}
\DeclareMathOperator{\erfc}{erfc}

\DeclareMathOperator{\csch}{csch}

\usepackage[draft]{todonotes} 

\title{Expected area of the star hull of planar Brownian motion and bridge}

\author{
Hugo Panzo
\\ 
\href{mailto:hugo.panzo@slu.edu}{\texttt{{\small hugo.panzo@slu.edu}}}
}

\date{\today}

\begin{document}

\maketitle

\begin{abstract}
We study the \emph{star hull} of planar Brownian motion and bridge, and relate this random compact set to the more familiar convex and topological hulls. Roughly speaking, the star hull is the smallest starshaped set (with respect to the origin) that contains the trace of the path. In particular, we prove that the expected areas of the star hulls are $\frac{3\pi}{8}$ and $\frac{\pi}{4}$ for planar Brownian motion and bridge, respectively. Along the way, we find the one-point marginal distribution of the radial functions of both traces. Our proofs rely on a detailed analysis of the first hitting time and place of a half-line by planar Brownian motion, and one of our main results is a remarkably simple expression for the Laplace transform of this joint law.
\end{abstract}

\begin{keywords}
Bessel bridge; convex hull; planar Brownian bridge; planar Brownian motion; star hull; topological hull.
\end{keywords}

\begin{subjclass}
Primary 60D05, 60J65; Secondary 60E10.
\end{subjclass}

\section{Introduction and main results}

The \emph{convex hull} of planar Brownian motion has been studied since the 1940s, beginning with the work of P.~L\'evy \cite{Levy}. Since then a rich literature has developed around quantitative and geometric properties of the convex hull, for example, exact formulas for the expected perimeter \cite{Takacs_solution} and area \cite{El_Bachir}, smoothness \cite{smooth} and curvature \cite{curvature} properties, laws of the iterated logarithm \cite{Khoshnevisan}, large deviations \cite{LDP}, and moment bounds on other geometric functionals and related inverse processes \cite{McRedmond_Xu, Jovalekic, hull_bounds}. The reader is encouraged to consult the references cited in these works as well as the survey \cite{Majumdar} for further results and background.

Alongside the convex hull, another fundamental object in the planar Brownian literature is the \emph{topological hull} of the path, also known as the \emph{Brownian hull}; it is the union of the trace with all of the bounded connected components of its complement in the plane. This set has connected complement and its boundary is the \emph{outer boundary} or \emph{frontier} of the Brownian motion. Deep connections between planar Brownian motion and Schramm--Loewner evolution were established by Lawler--Schramm--Werner, who showed in \cite{Lawler_Schramm_Werner} that, loosely speaking, the outer boundary of planar Brownian motion is chordal $\mathrm{SLE}_{8/3}$. These SLE techniques were later used to compute the expected area of the topological hull of the planar Brownian bridge; see \cite{Garban_Ferreras} and Theorem \ref{thm:bridge_areas} below.

In the present paper we bring a third hull into the probabilistic picture, namely, the \emph{star hull}. Given a set $\mathcal{A} \subset \mathbb{R}^2$, the star hull of $\mathcal{A}$, which we denote by $\strh(\mathcal{A})$, is the smallest starshaped set (with respect to the origin) that contains $\mathcal{A}$. In other words, $\strh(\mathcal{A})$ is obtained from $\mathcal{A}$ by filling in the lines of sight from the origin to all points of $\mathcal{A}$. While starshaped sets are classical in analysis and geometry \cite{starshaped}, the star hull seems to be much less studied; see \cite{Li, Klain1, Klain2, min_area_hull, star_hull, Kiderlen} for the only references we could find that deal with star hulls. In particular, to our knowledge the star hull has not previously appeared in the literature in connection with Brownian motion. However, following the posting of this preprint to the arXiv, the star hull of planar Brownian motion stopped upon exiting the unit disk has since been investigated in \cite{CH_until_exit}.

Let $\mathcal{K}\subset\mathbb{R}^2$ be a compact set containing the origin. We denote by $\hull(\mathcal{K})$, $\strh(\mathcal{K})$, and $\conv(\mathcal{K})$ the topological, star, and convex hulls of $\mathcal{K}$, respectively. It is not difficult to show that all three of these hulls are compact sets. Moreover, we prove in Lemma \ref{lem:inclusion_chain} below that the following inclusion chain holds
\begin{equation}\label{eq:inclusion_chain}
\mathcal{K} \subset \hull(\mathcal{K}) \subset \strh(\mathcal{K}) \subset \conv(\mathcal{K});
\end{equation}
see Section \ref{sec:three_hulls} for proofs of these claims along with formal definitions of the hulls. 

Planar Brownian motion $(\boldsymbol{W}_t:t\geq 0)$ starting at $\boldsymbol{0}$ and planar Brownian bridge $(\boldsymbol{B}_t:0\leq t\leq 1)$ from $\boldsymbol{0}$ to $\boldsymbol{0}$ provide us with two random compact sets that contain the origin, namely, their traces $\boldsymbol{W}_{[0,1]}$ and $\boldsymbol{B}_{[0,1]}$. From these sets we obtain the six random hulls
\begin{alignat}{3}
\mathcal{T}^{\mathrm{BM}} &:= \hull(\boldsymbol{W}_{[0,1]}), &\quad
\mathcal{S}^{\mathrm{BM}} &:= \strh(\boldsymbol{W}_{[0,1]}), &\quad
\mathcal{C}^{\mathrm{BM}} &:= \conv(\boldsymbol{W}_{[0,1]}),\label{eq:BM_hulls}\\
\nonumber\\
\mathcal{T}^{\mathrm{BB}} &:= \hull(\boldsymbol{B}_{[0,1]}), &\quad
\mathcal{S}^{\mathrm{BB}} &:= \strh(\boldsymbol{B}_{[0,1]}), &\quad
\mathcal{C}^{\mathrm{BB}} &:= \conv(\boldsymbol{B}_{[0,1]}).\label{eq:bridge_hulls}
\end{alignat}
Figure \ref{fig:illustrations} depicts all three hulls drawn for a random walk approximation to planar Brownian motion.


\begin{figure}[htbp]
    \centering
    \begin{subfigure}[b]{0.49\textwidth}
        \centering
        \includegraphics[width=\linewidth]{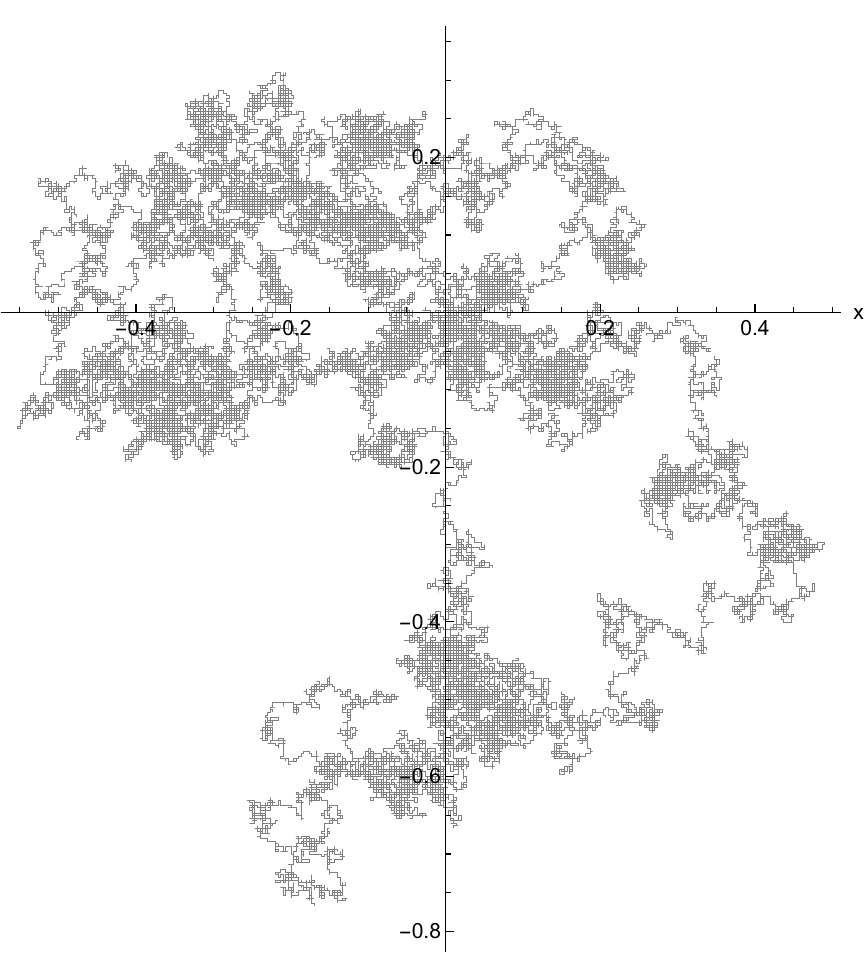}
        \caption{Trajectory of a scaled random walk.}
    \end{subfigure}
    \hfill
    \begin{subfigure}[b]{0.49\textwidth}
        \centering
        \includegraphics[width=\linewidth]{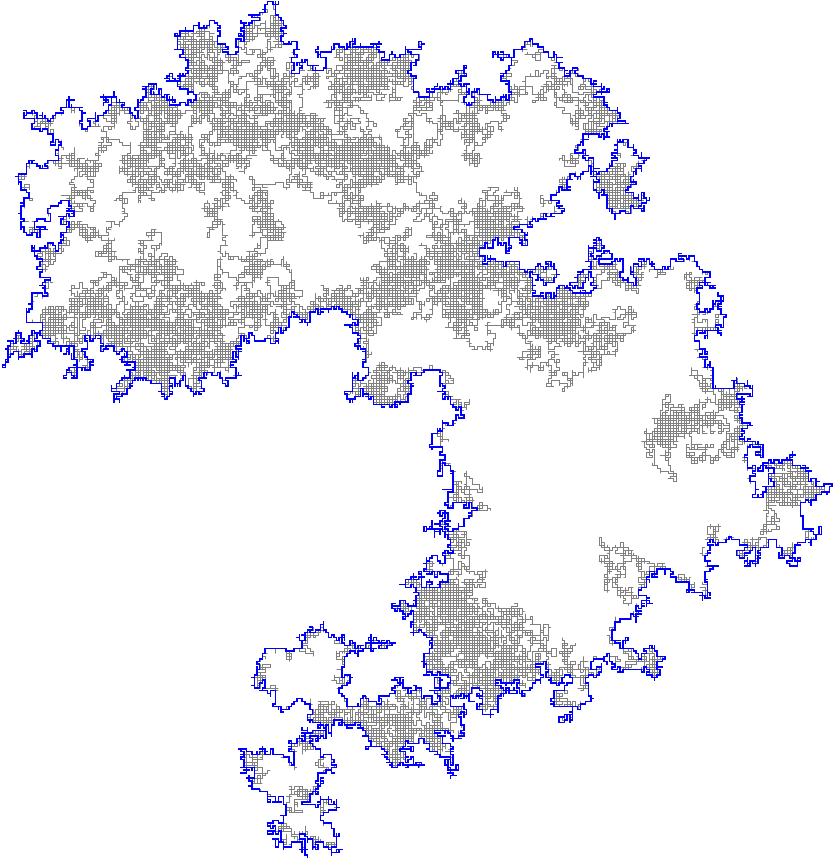}
        \caption{Topological hull of the same trajectory.}
    \end{subfigure}
    
    \vspace{1em} 
    
    \begin{subfigure}[b]{0.49\textwidth}
        \centering
        \includegraphics[width=\linewidth]{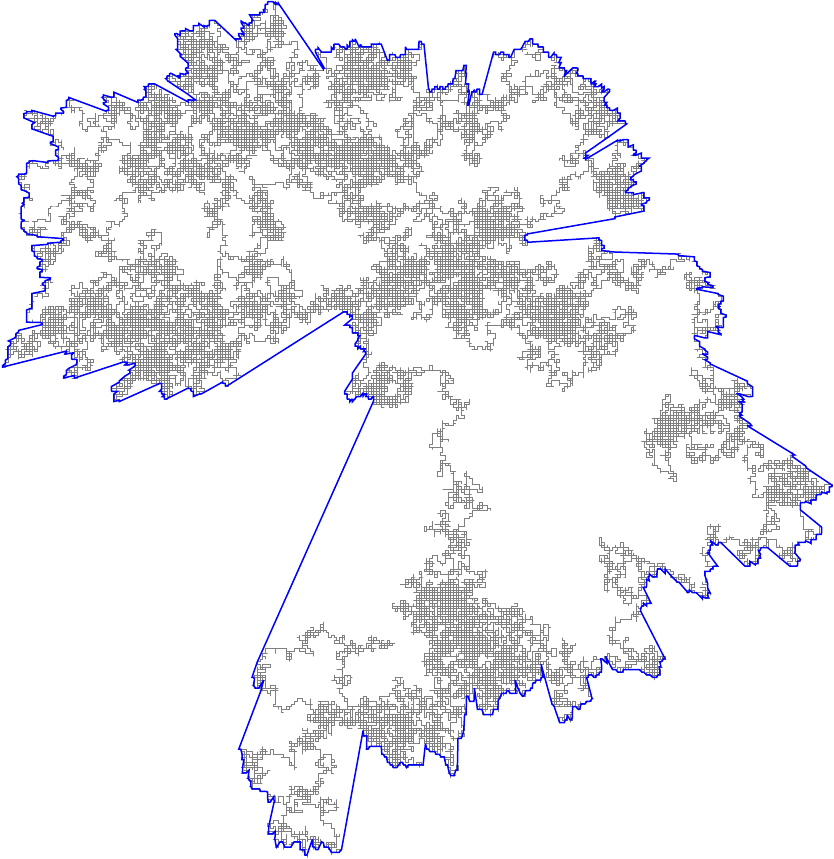}
        \caption{Star hull of the same trajectory.}
    \end{subfigure}
    \hfill
    \begin{subfigure}[b]{0.49\textwidth}
        \centering
        \includegraphics[width=\linewidth]{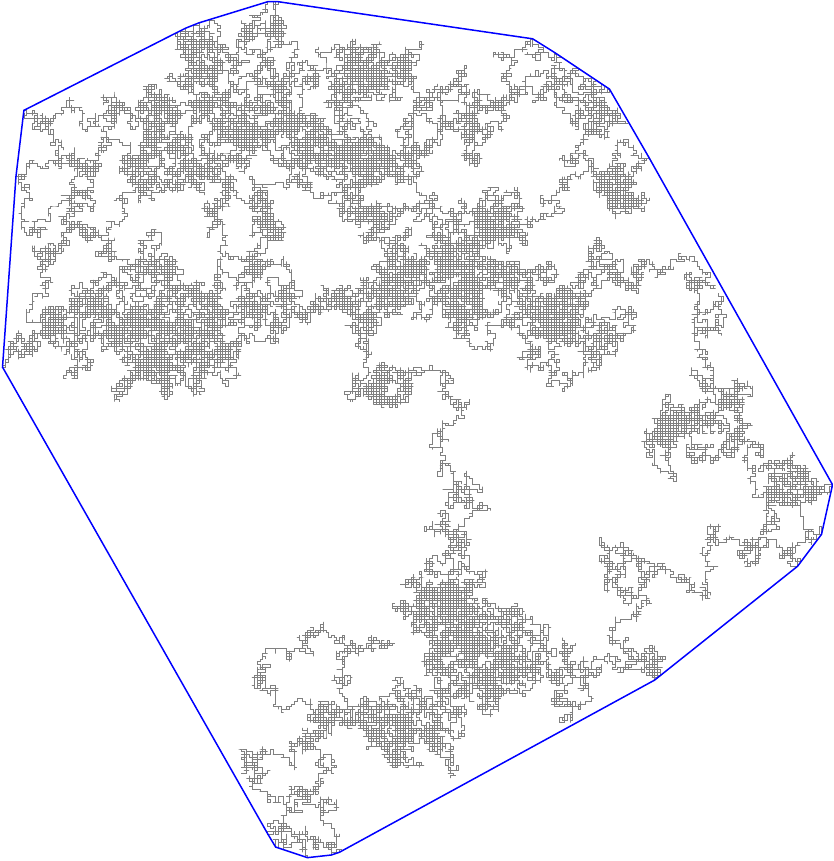}
        \caption{Convex hull of the same trajectory.}
    \end{subfigure}
    
    \caption{Illustrations of a single trajectory of a scaled $\mathbb{Z}^2$ nearest-neighbor random walk approximation to planar Brownian motion starting at the origin, together with each of the three kinds of hulls drawn for the same trajectory.}
    \label{fig:illustrations}
\end{figure}


Our first main result concerns the expected areas of the three bridge hulls.
\begin{theorem}\label{thm:bridge_areas}
Consider the planar Brownian bridge $(\boldsymbol{B}_t:0\leq t\leq 1)$ from $\boldsymbol{0}$ to $\boldsymbol{0}$, and let the corresponding topological hull $\mathcal{T}^{\mathrm{BB}}$, star hull $\mathcal{S}^{\mathrm{BB}}$, and convex hull $\mathcal{C}^{\mathrm{BB}}$ be defined as in \eqref{eq:bridge_hulls} and Section \ref{sec:three_hulls}. Then the expected areas of these hulls are 
\[
\mathbb{E}\big[\mathrm{area}(\mathcal{T}^{\mathrm{BB}})\big] = \frac{\pi}{5},
~~~~
\mathbb{E}\big[\mathrm{area}(\mathcal{S}^{\mathrm{BB}})\big] = \frac{\pi}{4},
~~~~
\mathbb{E}\big[\mathrm{area}(\mathcal{C}^{\mathrm{BB}})\big] = \frac{\pi}{3}.
\]
\end{theorem}

The outer terms in this remarkable sequence of expected hull areas come from \cite{Garban_Ferreras} and \cite[Equation (144)]{Majumdar}, respectively, while the middle identity is new and is proven in Section \ref{sec:areas}. One wonders if the harmonic progression is merely coincidental.

Our next result concerns the expected areas of the three planar Brownian motion hulls. While the expected area of the convex hull was computed more than 40 years ago in \cite{El_Bachir}, it seems that the expected area of the topological hull is still unknown; see Open Question \ref{Q:top_area} for more details. However, our exact formula for the expected area of the star hull does provide an improved upper bound; see Section \ref{sec:areas} for the proof.

\begin{theorem}\label{thm:BM_star_area}
Consider the planar Brownian motion $(\boldsymbol{W}_t:0\leq t\leq 1)$ starting at $\boldsymbol{0}$, and let the corresponding topological hull $\mathcal{T}^{\mathrm{BM}}$, star hull $\mathcal{S}^{\mathrm{BM}}$, and convex hull $\mathcal{C}^{\mathrm{BM}}$ be defined as in \eqref{eq:BM_hulls} and Section \ref{sec:three_hulls}. Then we have the expected areas
\[
\mathbb{E}\big[\mathrm{area}(\mathcal{S}^{\mathrm{BM}})\big] = \frac{3\pi}{8},
  \quad
  \mathbb{E}\big[\mathrm{area}(\mathcal{C}^{\mathrm{BM}})\big] = \frac{\pi}{2}.
\]
Consequently, the expected area of the topological hull has the upper bound
\[
\mathbb{E}\big[\mathrm{area}(\mathcal{T}^{\mathrm{BM}})\big]\leq \frac{3\pi}{8}.
\]
\end{theorem}

\subsection*{First hitting time and place of a ray}

The main results of this paper are based on a careful study of the first hitting time and place of a horizontal ray by planar Brownian motion. For $\rho\geq 0$, define 
\begin{equation}\label{eq:ray}
\mathcal{R}_\rho:=[\rho,\infty)\times\{0\}.
\end{equation}
The ray $\mathcal{R}_\rho$ is the part of the horizontal axis to the right of and including the point $(\rho,0)$. The first hitting time of $\mathcal{R}_\rho$ is defined by
\begin{equation}\label{eq:first_time}
T_\rho:=\inf\{t \geq 0 : \boldsymbol{W}_t \in \mathcal{R}_\rho\}.
\end{equation}
It is well known that planar Brownian motion hits sets with positive logarithmic capacity, so $T_\rho$ is finite almost surely. Recall that $\boldsymbol{W}_t=(W_t^{(1)},W_t^{(2)})$, where for each $i=1,2$, the $i$th coordinate process $(W_t^{(i)}:t\geq 0)$ is a one-dimensional Brownian motion starting at $0$ that is independent of the other coordinate. With this representation, we can define the first hitting place of $\mathcal{R}_\rho$ by
\begin{equation}\label{eq:first_place}
X_\rho:=W_{T_\rho}^{(1)}.
\end{equation}

Our next result is a remarkably simple expression for the joint Laplace transform of $(T_\rho,X_\rho)$; see Section \ref{sec:joint_Laplace} for the proof.

\begin{theorem}\label{thm:joint_Laplace}
Consider the first hitting time and place $(T_\rho,X_\rho)$ of the horizontal ray $\mathcal{R}_\rho$ by the planar Brownian motion $(\boldsymbol{W}_t:t\geq 0)$ starting at $\boldsymbol{0}$, as defined by \eqref{eq:ray}, \eqref{eq:first_time}, and \eqref{eq:first_place}. Then for every $\rho>0$ and $\lambda,\mu \geq 0$, we have
\[
\mathbb{E}[\exp(-\lambda T_\rho - \mu X_\rho)]
=\erfc\mleft(\sqrt{\rho\mleft(\sqrt{2\lambda}+\mu \mright)}\mright),
\]
where $\erfc(z):=\frac{2}{\sqrt{\pi}}\int_z^\infty e^{-u^2}\D{u}$ is the complementary error function.
\end{theorem}

It is worth mentioning that a trivariate Laplace transform which incorporates the local time at zero of the vertical coordinate in addition to $T_\rho$ and $X_\rho$ was obtained by Isozaki using Wiener--Hopf factorization in \cite{Isozaki}. However, even when specialized to the bivariate case treated in our Theorem \ref{thm:joint_Laplace}, the expression obtained by Isozaki is rather complicated and involves the integral of a complex-valued auxiliary function which is itself defined as an integral. In particular, it seems difficult to recover from \cite[Theorem 1.1]{Isozaki} the following conditional structure of $(T_\rho,X_\rho)$ that can be readily deduced from Theorem \ref{thm:joint_Laplace}; see Section \ref{sec:conditional_structure} for the proof. 

\begin{theorem}\label{thm:conditional_structure}
Let $T_\rho$ and $X_\rho$ be as in Theorem \ref{thm:joint_Laplace}. Then for any $\rho>0$ and $x\geq\rho$, conditionally on $X_\rho=x$, the random variable $T_\rho$ has the law of the first passage time to the level $x$ of one-dimensional Brownian motion starting at $0$. Equivalently, for every $\lambda \geq 0$, we have 
\[
\mathbb{E}\big[e^{-\lambda T_\rho} \big|X_\rho\big] = e^{-\sqrt{2\lambda}X_\rho}~\text{almost surely}.
\]
\end{theorem}

While our joint Laplace transform formula is new, this conditional structure, as well as the corresponding joint and marginal density formulas for $(T_\rho,X_\rho)$ from Theorem \ref{thm:densities} below, first appeared in the as-yet-unpublished preprint \cite{Wroclaw}, where they were derived using markedly different methods. See Section \ref{sec:densities} for our derivation of the densities. In order to state these density formulas, we first need to recall the definition of the modified Bessel function of the second kind. This function, also known as the Bessel \emph{K} function of order $\nu\in\mathbb{C}$, has the representation
\[
K_\nu(z)=\frac{z^\nu}{2^{\nu+1}}\int_0^\infty t^{-\nu-1}e^{-t-\frac{z^2}{4t}}\D{t};
\]
see \cite[Chapter 10]{DLMF}. For $\nu>0$, this reference reports real argument asymptotics of
\begin{equation}\label{eq:BesselK_asymptotic}
K_\nu(x)\sim 2^{\nu-1}\Gamma(\nu)\, x^{-\nu}~\text{ as }~x\searrow 0,~~~~K_\nu(x)\sim \sqrt{\frac{\pi}{2x}}e^{-x}~\text{ as }~x\to\infty.
\end{equation}

\begin{theorem}\label{thm:densities}
For $\rho>0$, let $T_\rho$ and $X_\rho$ be as in Theorem \ref{thm:joint_Laplace}. Then we have that:
\begin{enumerate}[label=(\roman*)]
\item The joint density of $(T_\rho,X_\rho)$ is
\begin{equation}\label{eq:joint_density}
\mathbb{P}(T_\rho\in\D{t},\,X_\rho\in \D{x}) = \frac{\sqrt{\rho}\,e^{-\frac{x^2}{2t}}}{\sqrt{2\pi^3 t^3 (x-\rho) }}\D{t}\D{x}, ~~ t>0,\, x>\rho;
\end{equation}

\item The marginal density of $T_\rho$ is
\begin{equation}\label{eq:T_density}
\mathbb{P}(T_\rho\in\D{t}) = \frac{\rho\,e^{-\frac{\rho^2}{4t}}}{2\sqrt{\pi^3 t^3}}K_\frac{1}{4}\mleft(\frac{\rho^2}{4t}\mright)\D{t}, ~~ t>0;
\end{equation}

\item The marginal density of $X_\rho$ is
\begin{equation}\label{eq:X_density}
\mathbb{P}(X_\rho\in \D{x}) = \frac{\sqrt{\rho}}{\pi x\sqrt{x-\rho}}\D{x}, ~~ x>\rho.
\end{equation}
\end{enumerate}
\end{theorem}

\begin{remark}\label{rem:one_fourth}
The small argument asymptotic from \eqref{eq:BesselK_asymptotic} can be used with the Legendre duplication formula to recover from \eqref{eq:T_density} the right-tail asymptotic of Isozaki
\begin{equation}
\mathbb{P}(T_\rho>t)\sim\frac{2^{5/4}\sqrt{\rho}}{\sqrt{\pi}\,\Gamma(3/4)}t^{-1/4}~\text{ as }~t\to\infty;\label{eq:one_fourth}
\end{equation}
see \cite[Corollary 1.1]{Isozaki}. The $t^{-1/4}$ decay rate in \eqref{eq:one_fourth} can be surmised by recasting the hitting time of a ray as the exit time of a degenerate wedge with total opening angle $2\pi$ and then appealing to Spitzer’s well-known integrability condition for the exit time of wedges \cite[Theorem 2]{Spitzer}. The same decay rate in the simple random walk setting was first observed by Lawler \cite[Equation (2.35)]{intersections}; see also \cite{slit_plane, Fukai}.
\end{remark}

\subsection*{Radial distance of the trace}

Define the \emph{radial distance} of the traces of planar Brownian motion and bridge by 
\begin{equation}\label{eq:radial_distance}
\begin{split}
R^{\mathrm{BM}} &:= \sup\mleft\{\rho \geq 0 : (\rho,0) \in \boldsymbol{W}_{[0,1]}\mright\},\\
R^{\mathrm{BB}} &:= \sup\mleft\{\rho \geq 0 : (\rho,0) \in \boldsymbol{B}_{[0,1]}\mright\}.
\end{split}
\end{equation}
In other words, the radial distance is the distance from the origin to the rightmost point in the intersection of the trace with the horizontal axis. Alternatively, it is nothing but the one-point marginal of the corresponding \emph{radial function} in the $\theta=0$ direction; see Section \ref{sec:radial}. Since the paths of Brownian motion starting at $\boldsymbol{0}$ and Brownian bridge starting and ending at $\boldsymbol{0}$ both have distributions which are invariant under rotations about the origin, it follows that the radial functions of their traces are stationary random fields on $[0,2\pi)$.  

The random variable $R^{\mathrm{BM}}$ and the point which attains it, as well as more general versions of $R^{\mathrm{BM}}$ obtained by intersecting the trace with a vertically shifted horizontal line, have appeared in the study of cone and twist points of planar Brownian motion by Burdzy and Le Gall; see \cite{Burdzy} and Sections III.4 and IV.4 of \cite{Le_Gall}. However, as far as the present author knows, neither the distributions nor any moments of $R^{\mathrm{BM}}$ or $R^{\mathrm{BB}}$ have been computed in the literature. The following theorem identifies the densities and moments of both of these random variables; see Section \ref{sec:radial_distance} for the proof. This theorem, when paired with Lemma \ref{lem:mean_area_formula} below, is what allows us to compute the expected areas of the star hulls in Theorems \ref{thm:bridge_areas} and \ref{thm:BM_star_area}.

\begin{theorem}\label{thm:radial_distance}
Let the radial distances $R^{\mathrm{BM}}$ and $R^{\mathrm{BB}}$ be defined as in \eqref{eq:radial_distance}.
\begin{enumerate}[label=(\roman*)]
\item We have the equality in distribution $R^{\mathrm{BM}}\stackrel{d}{=} T_1^{-\frac{1}{2}}$, where $T_1$ denotes the first hitting time of the ray $\mathcal{R}_1$ defined by \eqref{eq:ray} and \eqref{eq:first_time}. Furthermore, $R^{\mathrm{BM}}$ has the density and fractional moments given by 
\begin{align*}
\mathbb{P}\mleft(R^{\mathrm{BM}}\in \D{\rho}\mright) &= \frac{e^{-\frac{\rho^2}{4}}}{\sqrt{\pi^3}}K_\frac{1}{4}\mleft(\frac{\rho^2}{4}\mright)\D{\rho}, ~~ \rho>0;\\
\\
\mathbb{E}\mleft[\mleft(R^{\mathrm{BM}}\mright)^p\mright]&=\frac{\Gamma\mleft(p+\frac{1}{2}\mright)}{\sqrt{2^p \pi}\,\Gamma\mleft(\frac{p}{2}+1\mright)},~~p>-\frac{1}{2}.
\end{align*}

\item We have the equality in distribution $R^{\mathrm{BB}}\stackrel{d}{=} \frac{1}{2}|Z|$, where $Z$ is a standard normal random variable. Consequently, $R^{\mathrm{BB}}$ has the half-normal distribution with density and fractional moments given by
\begin{align*}
\mathbb{P}\mleft(R^{\mathrm{BB}}\in \D{\rho}\mright)&= 2\sqrt{\frac{2}{\pi}}e^{-2\rho^2}\D{\rho}, ~~ \rho>0;\\
\\
\mathbb{E}\mleft[\mleft(R^{\mathrm{BB}}\mright)^p\mright]&=\frac{\Gamma\mleft(\frac{p}{2}+\frac{1}{2}\mright)}{\sqrt{2^p \pi}},~~p>-1.
\end{align*}
\end{enumerate}
\end{theorem}

\begin{remark}
In principle, the interpretation of $T_1$ as the exit time from a degenerate wedge mentioned in Remark \ref{rem:one_fourth} should allow for the calculation of moments of $R^{\mathrm{BM}}$ via the Gauss--Laplace transform for the wedge exit time found in \cite[Proposition 1.1]{cone_exit}. However, using this approach without going through the ordinary Laplace transform as we do, would require repeated differentiation which seems difficult to carry out in practice.
\end{remark}

\begin{remark}
In Janson’s terminology from \cite{Janson}, both $R^{\mathrm{BM}}$ and $R^{\mathrm{BB}}$ are examples of random variables with \emph{moments of gamma type}. 
\end{remark}




\subsection{Organization of the paper}
The rest of the paper is organized as follows. We introduce the three planar hulls in Section \ref{sec:three_hulls} and discuss their properties, with the star hull and radial functions covered in Section \ref{sec:radial} and the other hulls covered in Section \ref{sec:other_hulls}. Section \ref{sec:preliminaries} gathers some probabilistic preliminaries such as Brownian scaling and space-time transformations between Brownian motion and Brownian bridge in Section \ref{sec:scaling}, conformal invariance and time change in Section \ref{sec:conformal_invariance}, explicit formulas for the Laplace transform of quadratic functionals of Bessel bridges in Section \ref{sec:quad_functionals}, and some results involving the complementary error function in Section \ref{sec:erfc}. Theorem \ref{thm:joint_Laplace} is proven in Section \ref{sec:joint_Laplace}, Theorem \ref{thm:conditional_structure} in Section \ref{sec:conditional_structure}, Theorem \ref{thm:densities} in Section \ref{sec:densities}, and Theorems \ref{thm:bridge_areas} and \ref{thm:BM_star_area} are proven in Section \ref{sec:areas}. Finally, some open questions and further directions for study are discussed in Section \ref{sec:open}.


\section{Three kinds of hulls}\label{sec:three_hulls}

In this section we recall the definitions of the three kinds of hulls studied in this paper and prove several properties such as the inclusion chain \eqref{eq:inclusion_chain}. We stress that these are by no means the only hulls of Brownian motion that one can consider. For an example of a different type of hull, see Section 1.6 of \cite{Sobolev_balls} for computations of the expected intrinsic volumes of the \emph{zonoid hull} of $d$-dimensional Brownian motion and bridge.

Throughout this section, uppercase letters in $\mathcal{MATHCAL}$ font are used to denote subsets of $\mathbb{R}^2$, typically with a certain property indicated by the particular choice of letter. For example, we use $\mathcal{A}$ for arbitrary sets, $\mathcal{D}$ for closed disks, $\mathcal{K}$ for compact sets which may or may not contain the origin, $\mathcal{S}$ for starshaped sets, and $\mathcal{U}$ for open sets. For two points $\boldsymbol{x},\boldsymbol{y} \in \mathbb{R}^2$, we write
\[
[\boldsymbol{x}:\boldsymbol{y}] := \{\lambda \boldsymbol{x} + (1-\lambda)\boldsymbol{y} : 0 \leq \lambda \leq 1\}
\]
for the line segment joining $\boldsymbol{x}$ and $\boldsymbol{y}$. Note that $[\boldsymbol{x}:\boldsymbol{y}]=[\boldsymbol{y}:\boldsymbol{x}]$, and that $[\boldsymbol{x}:\boldsymbol{x}]$ equals the singleton $\{\boldsymbol{x}\}$. Recall that a set $\mathcal{S}\subset \mathbb{R}^2$ is said to be \emph{starshaped} with respect to a point $\boldsymbol{p}\in \mathcal{S}$ if $[\boldsymbol{p}:\boldsymbol{x}]\subset \mathcal{S}$ for all $\boldsymbol{x}\in \mathcal{S}$. The reader can refer to \cite{starshaped} for an extensive survey on starshaped sets. 


\subsection{Star hulls and radial functions}\label{sec:radial}

\begin{definition}\label{def:star_hull}
The \textbf{star hull} of a set $\mathcal{A}\subset\mathbb{R}^2$ with respect to the origin is the set
\begin{align*}
\strh(\mathcal{A})&:=\bigcup_{\boldsymbol{x} \in \mathcal{A}} [\boldsymbol{0} : \boldsymbol{x}]\\
&=\{\lambda\boldsymbol{x}:0\leq \lambda\leq 1\text{ and }\boldsymbol{x}\in\mathcal{A}\}.
\end{align*}
\end{definition}

It is clear that $\strh(\mathcal{A})$ is starshaped with respect to $\boldsymbol{0}$ and that $\mathcal{A}\subset\strh(\mathcal{A})$. Moreover, it is straightforward to check that if $\mathcal{S}$ is any starshaped set with respect to $\boldsymbol{0}$ that contains $\mathcal{A}$, then $\mathcal{S}$ contains $\strh(\mathcal{A})$ as well. In this sense, $\strh(\mathcal{A})$ is the smallest starshaped set with respect to the origin that contains $\mathcal{A}$. The next proposition shows that the star hull preserves compactness.

\begin{proposition}\label{prop:SH_compact}
If $\mathcal{K} \subset \mathbb{R}^2$ is compact, then $\strh(\mathcal{K})$ is also compact.
\end{proposition}

\begin{proof}
Consider the continuous map $f : [0,1]\times \mathbb{R}^2 \to \mathbb{R}^2$ given by $f(\lambda,\boldsymbol{x}) = \lambda \boldsymbol{x}$. Since $[0,1]\times\mathcal{K}$ is compact, the image $f([0,1]\times\mathcal{K})=\strh(\mathcal{K})$ is also compact.
\end{proof}

The relevant characteristics of the star hull of a planar set are encoded by the \emph{radial function} of that set. For $\theta \in [0,2\pi)$, the radial function of $\mathcal{A}\subset\mathbb{R}^2$ is defined by
\begin{equation}\label{eq:radial_function}
r_\mathcal{A}(\theta) := \sup\{\rho \geq 0 : (\rho\cos\theta,\rho\sin\theta) \in \mathcal{A}\},
\end{equation}
with the convention that $\sup\{\varnothing\}=0$. In the following proposition, we list some useful properties of the radial function of a compact set containing the origin.

\begin{proposition}\label{prop:radial_function}
Let $\mathcal{K} \subset \mathbb{R}^2$ be compact with $\boldsymbol{0} \in \mathcal{K}$. Then we have:
\begin{enumerate}[label=(\roman*)]
\item $r_\mathcal{K}(\theta) < \infty$ for all $\theta \in [0,2\pi)$;
\item $\big(r_\mathcal{K}(\theta)\cos\theta,r_\mathcal{K}(\theta)\sin\theta\big)\in \mathcal{K}$ for all $\theta \in [0,2\pi)$;
\item $r_\mathcal{K}$ is upper semicontinuous on $[0,2\pi)$;
\item $r_{\strh(\mathcal{K})}(\theta) = r_\mathcal{K}(\theta)$ for all $\theta \in [0,2\pi)$.
\end{enumerate}
\end{proposition}

\begin{proof}
Property \emph{(i)} is an immediate consequence of the boundedness of $\mathcal{K}$. 

Property \emph{(ii)} follows from the fact that for each $\theta \in [0,2\pi)$, the supremum appearing in \eqref{eq:radial_function} is always attained by some $\rho\geq 0$ since $\boldsymbol{0}\in\mathcal{K}$ and $\mathcal{K}$ is compact. 

Property \emph{(iii)} requires a little more work to prove than the previous properties but is still straightforward. First fix $\theta\in [0,2\pi)$, and let $\{\theta_n\}_{n\geq 1}$ be any sequence of angles converging to $\theta$. Put $L(\theta)=\limsup_{n\to\infty}r_\mathcal{K}(\theta_n)$, and pick a subsequence $\{\theta_{n_k}\}_{k\geq 1}$ with $\lim_{k\to\infty}r_\mathcal{K}(\theta_{n_k})=L(\theta)$. We know from Property \emph{(ii)} that the sequence of points $\{\boldsymbol{x}_k\}_{k\geq 1}$ defined by 
\[
\boldsymbol{x}_k:=(r_\mathcal{K}(\theta_{n_k})\cos\theta_{n_k},r_\mathcal{K}(\theta_{n_k})\sin\theta_{n_k})
\]
is contained in $\mathcal{K}$. Moreover, we have $\lim_{k\to\infty}\boldsymbol{x}_k=(L(\theta)\cos\theta,L(\theta)\sin\theta)=:\boldsymbol{x}$ with $\boldsymbol{x}\in\mathcal{K}$ being another consequence of compactness. Lastly, since the point $(L(\theta)\cos\theta,L(\theta)\sin\theta)$ is an element of $\mathcal{K}$, we get from the supremum in the definition \eqref{eq:radial_function} that $L(\theta)\leq r_\mathcal{K}(\theta)$. Therefore, $\limsup_{n\to\infty}r_\mathcal{K}(\theta_n)\leq r_\mathcal{K}(\theta)$ as required to conclude the upper semicontinuity of $r_\mathcal{K}$ at $\theta$.

Property \emph{(iv)} is proved by showing that each radial function is bounded below by the other. Clearly $r_{\strh(\mathcal{K})}(\theta)\geq r_{\mathcal{K}}(\theta)$ holds for all $\theta \in [0,2\pi)$ since $\strh(\mathcal{K})\supset\mathcal{K}$. Note that the reverse inequality holds trivially if $r_{\strh(\mathcal{K})}(\theta)=0$, so we can assume that $r_{\strh(\mathcal{K})}(\theta)>0$. For each $\theta \in [0,2\pi)$ where this assumption holds, consider the point $\boldsymbol{x}_\theta\neq \boldsymbol{0}$ defined by
\[
\boldsymbol{x}_\theta:=\big(r_{\strh(\mathcal{K})}(\theta)\cos\theta,r_{\strh(\mathcal{K})}(\theta)\sin\theta\big).
\]
We can use Property \emph{(ii)} to deduce that $\boldsymbol{x}_\theta\in\strh(\mathcal{K})$. Now it follows from Definition \ref{def:star_hull} that $\rho_\theta^{-1}\boldsymbol{x}_\theta\in \mathcal{K}$ for some $0<\rho_\theta\leq 1$. This allows us to write 
\[
r_{\mathcal{K}}(\theta)\geq \mleft|\rho_\theta^{-1}\boldsymbol{x}_\theta\mright|=\rho_\theta^{-1}r_{\strh(\mathcal{K})}(\theta)\geq r_{\strh(\mathcal{K})}(\theta).
\]
\end{proof}

The next proposition expresses the area of the star hull of a set in terms of the radial function of the set.

\begin{proposition}\label{prop:area_formula}
Let $\mathcal{K} \subset \mathbb{R}^2$ be compact with $\boldsymbol{0} \in \mathcal{K}$. Then
\[
\mathrm{area}\big(\strh(\mathcal{K})\big)=\frac{1}{2}\int_0^{2\pi} r_\mathcal{K}(\theta)^2\D{\theta}.
\]
\end{proposition}

\begin{proof}
For each $\theta\in [0,2\pi)$, the radial sections $\strh(\mathcal{K}) \cap \{(\rho\cos\theta,\rho\sin\theta):\rho\geq 0\}$ of $\strh(\mathcal{K})$ are precisely the line segments $[\boldsymbol{0}:\big(r_{\strh(\mathcal{K})}(\theta)\cos\theta,r_{\strh(\mathcal{K})}(\theta)\sin\theta\big)]$. Since $\strh(\mathcal{K})$ is a Borel set, we can use Tonelli's theorem along with polar coordinates and Property \emph{(iv)} of Proposition \ref{prop:radial_function} to write
\begin{align*}
\mathrm{area}\big(\strh(\mathcal{K})\big)=\int_{\mathbb{R}^2}\mathbbm{1}_{\strh(\mathcal{K})}(\boldsymbol{x})\D{\boldsymbol{x}}&=\int_0^{2\pi}\int_0^\infty \mathbbm{1}_{\strh(\mathcal{K})}(\rho\cos\theta,\rho\sin\theta)\rho\D{\rho}\D{\theta}\\
&=\int_0^{2\pi}\int_0^{r_{\strh(\mathcal{K})}(\theta)} \rho \D{\rho}\D{\theta}\\
&=\frac{1}{2}\int_0^{2\pi} r_\mathcal{K}(\theta)^2\D{\theta}.
\end{align*}
\end{proof}

An important consequence of Proposition \ref{prop:area_formula} is that it leads to a simple formula for the expected area of the star hull of a rotationally invariant random continuous path in $\mathbb{R}^2$ that starts at the origin. In particular, this formula applies to planar Brownian motion and bridge. Note that a random star hull which is almost surely compact may also have infinite expected area.

\begin{lemma}\label{lem:mean_area_formula}
Let $(\boldsymbol{X}_t:0\leq t \leq 1)$ be a random path in $\mathbb{R}^2$ satisfying:
\begin{enumerate}[label=(\roman*)]
\item $\boldsymbol{X}_0=\boldsymbol{0}$ almost surely;
\item $(\boldsymbol{X}_t:0\leq t \leq 1)$ is continuous almost surely;
\item The law of $(\boldsymbol{X}_t:0\leq t \leq 1)$ is rotationally invariant, that is, 
\[
(\boldsymbol{X}_t:0\leq t \leq 1)\stackrel{d}{=}(\mathfrak{M}_\theta\,\boldsymbol{X}_t:0\leq t \leq 1)
\]
for any $\theta\in\mathbb{R}$ and counterclockwise rotation matrix 
\[
\mathfrak{M}_\theta=
\begin{bmatrix}
\cos\theta & -\sin\theta \\
\sin\theta & \cos\theta 
\end{bmatrix}.
\]
\end{enumerate}
Then the star hull of the trace of $(\boldsymbol{X}_t:0\leq t \leq 1)$ is almost surely compact and its expected area can be expressed in terms of the radial function $r_{\boldsymbol{X}_{[0,1]}}(\theta)$ as 
\[
\mathbb{E}\big[\mathrm{area}\big(\strh (\boldsymbol{X}_{[0,1]})\big)\big] = \pi\,\mathbb{E}\big[r_{\boldsymbol{X}_{[0,1]}}(0)^2\big].
\]
\end{lemma}

\begin{proof}
The almost sure compactness of $\strh (\boldsymbol{X}_{[0,1]})$ is a straightforward consequence of the continuity assumption and Proposition \ref{prop:SH_compact}. Moreover, the first two assumptions allow us to invoke Proposition \ref{prop:area_formula}, from which the expected area formula follows after an application of Tonelli's theorem together with the rotational invariance assumption.
\end{proof}


\subsection{Topological and convex hulls}\label{sec:other_hulls}

Recall that any open set $\mathcal{U}\subset\mathbb{R}^2$ is the union of at most countably many disjoint connected open sets. These are the so-called \emph{connected components} of $\mathcal{U}$.

\begin{definition}\label{def:top_hull}
The \textbf{topological hull} of a compact set $\mathcal{K}\subset\mathbb{R}^2$, denoted by $\hull(\mathcal{K})$, is the union of $\mathcal{K}$ together with all of the bounded connected components of $\mathbb{R}^2\setminus\mathcal{K}$. Equivalently, if $\mathcal{U}_\infty$ denotes the unique unbounded connected component of $\mathbb{R}^2\setminus\mathcal{K}$, then $\hull(\mathcal{K})=\mathbb{R}^2\setminus \mathcal{U}_\infty$.
\end{definition}

It is clear that $\mathcal{K}\subset\hull(\mathcal{K})$. The following proposition collects some other properties of the topological hull of a compact planar set. The last property shows that $\hull(\mathcal{K})$ is the smallest compact set containing $\mathcal{K}$ whose complement is connected. This gives a characterization of the topological hull that is closer in spirit to that of the star and convex hulls.

\begin{proposition}
Let $\mathcal{K} \subset \mathbb{R}^2$ be compact. Then we have:
\begin{enumerate}[label=(\roman*)]
\item $\mathbb{R}^2 \setminus \hull(\mathcal{K})$ is connected;
\item $\hull(\mathcal{K})$ is compact;
\item If $\mathcal{K}' \subset \mathbb{R}^2$ is any compact set containing $\mathcal{K}$ such that $\mathbb{R}^2\setminus \mathcal{K}'$ is connected, then $\mathcal{K}'$ contains $\hull(\mathcal{K})$ as well.
\end{enumerate}
\end{proposition}

\begin{proof}
Property \emph{(i)} is an immediate consequence of Definition \ref{def:top_hull}. 

To prove Property \emph{(ii)}, first note that $\hull(\mathcal{K})$ is closed in $\mathbb{R}^2$ since it is the complement of the open set $\mathcal{U}_\infty$. Next we show that $\hull(\mathcal{K})$ is bounded. To see that this is the case, note that since $\mathcal{K}$ is compact, $\mathcal{K}$ must be contained in some closed and bounded disk $\mathcal{D}$. Hence, $\mathbb{R}^2\setminus\mathcal{D}\subset \mathbb{R}^2\setminus\mathcal{K}$. However, $\mathbb{R}^2\setminus\mathcal{D}$ is both connected and unbounded, so it must lie completely within the unique unbounded connected component of $\mathbb{R}^2\setminus\mathcal{K}$. In other words, we have $\mathbb{R}^2\setminus\mathcal{D}\subset \mathcal{U}_\infty$. Thus, $\mathbb{R}^2\setminus \mathcal{U}_\infty\subset\mathcal{D}$, and it follows that $\hull(\mathcal{K})$ is bounded. Now the compactness of $\hull(\mathcal{K})$ is a consequence of the Heine--Borel theorem.

Property \emph{(iii)} is proved using a similar argument. First note that the inclusion $\mathbb{R}^2\setminus\mathcal{K}'\subset \mathbb{R}^2\setminus\mathcal{K}$ is an obvious consequence of $\mathcal{K}\subset\mathcal{K}'$. Together with the fact that $\mathbb{R}^2\setminus\mathcal{K}'$ is both connected and unbounded, this implies that $\mathbb{R}^2\setminus\mathcal{K}'\subset\mathcal{U}_\infty$. Therefore, $\mathbb{R}^2\setminus\mathcal{U}_\infty\subset\mathcal{K}'$, and Property \emph{(iii)} follows.
\end{proof}

The third kind of hull dealt with in this paper is the familiar convex hull. Since we are working in the plane, it is convenient to define the convex hull via Carath\'{e}odory's theorem; see \cite[Theorem 1.1.4]{Schneider}.
\begin{definition}\label{def:conv_hull}
The \textbf{convex hull} of a set $\mathcal{A}\subset\mathbb{R}^2$ is the set
\begin{equation*}
\conv(\mathcal{A}):=\left\{ \sum_{i=1}^3\lambda_i \boldsymbol{x}_i:\lambda_i\geq 0,\,\sum_{i=1}^3\lambda_i=1,\text{ and }\boldsymbol{x}_i\in\mathcal{A}\right\}.
\end{equation*}
\end{definition}

It is well known that the convex hull of a compact planar set is compact. Moreover, the convex hull of an arbitrary set $\mathcal{A}\subset\mathbb{R}^2$ is the intersection of all convex subsets of $\mathbb{R}^2$ which contain $\mathcal{A}$. In this sense, $\conv(\mathcal{A})$ is the smallest convex set that contains $\mathcal{A}$. See Section 1.1 of \cite{Schneider} for the proofs of these claims and other properties of convex hulls. 

The last result of this section establishes the inclusion chain \eqref{eq:inclusion_chain} mentioned in the introduction.

\begin{lemma}\label{lem:inclusion_chain}
Let $\mathcal{K} \subset \mathbb{R}^2$ be compact with $\boldsymbol{0} \in \mathcal{K}$. Then we have the inclusion chain
\[
\mathcal{K} \subset \hull(\mathcal{K}) \subset \strh(\mathcal{K}) \subset \conv(\mathcal{K}).
\]
\end{lemma}

\begin{proof}
The inclusion $\mathcal{K} \subset \hull(\mathcal{K})$ is immediate from Definition \ref{def:top_hull}, while the inclusion $\strh(\mathcal{K}) \subset \conv(\mathcal{K})$ follows from Definition \ref{def:conv_hull} by taking $\boldsymbol{x}_2=\boldsymbol{0}$ and $\lambda_3=0$. 

It remains to show that $\hull(\mathcal{K})\subset \strh(\mathcal{K})$. Towards this end, let $\boldsymbol{x} \in \hull(\mathcal{K})$. If $\boldsymbol{x} \in \mathcal{K}$ there is nothing to prove since $\mathcal{K}\subset\strh(\mathcal{K})$, so suppose that $\boldsymbol{x}\neq \boldsymbol{0}$ lies in some bounded connected component of $\mathbb{R}^2 \setminus \mathcal{K}$. Now define the ray
\[
\mathcal{R}_{\boldsymbol{x}} := \{\rho \boldsymbol{x} : \rho \geq 0\}
\]
from the origin through $\boldsymbol{x}$, and consider the point $\boldsymbol{y}\in\mathcal{R}_{\boldsymbol{x}}\cap\mathcal{K}$ that is farthest from the origin. Such a point exists since $\mathcal{R}_{\boldsymbol{x}}\cap\mathcal{K}$ is nonempty and compact. Moreover, we must have $|\boldsymbol{y}|>|\boldsymbol{x}|$. To see this, suppose instead that $|\boldsymbol{y}|\leq|\boldsymbol{x}|$. Since $\boldsymbol{y}$ was chosen to be the point of $\mathcal{R}_{\boldsymbol{x}}\cap\mathcal{K}$ farthest from the origin, this would imply that the ray $\{\rho \boldsymbol{x} : \rho \geq 1\}$ is disjoint from $\mathcal{K}$. Hence, this ray is contained in $\mathbb{R}^2\setminus\mathcal{K}$, is connected, contains $\boldsymbol{x}$, and is unbounded. It follows that the connected component of $\mathbb{R}^2\setminus\mathcal{K}$ containing $\boldsymbol{x}$ is unbounded, contradicting the assumption that $\boldsymbol{x} \in \hull(\mathcal{K})$. Therefore, $|\boldsymbol{y}|>|\boldsymbol{x}|$, and the desired inclusion can be deduced by observing that this implies $\boldsymbol{x}\in [\boldsymbol{0}:\boldsymbol{y}]\subset \strh(\mathcal{K})$.
\end{proof}


\section{Probabilistic preliminaries}\label{sec:preliminaries}

In this section we collect results on the main probabilistic tools used in our analysis: Brownian scaling and space-time transformations between Brownian motion and Brownian bridge in Section \ref{sec:scaling}, conformal invariance and time change in Section \ref{sec:conformal_invariance}, explicit formulas for the Laplace transform of quadratic functionals of Bessel bridges in Section \ref{sec:quad_functionals}, and some results involving the complementary error function in Section \ref{sec:erfc}.


\subsection{Scaling and space-time transformations}\label{sec:scaling}

The familiar scaling property of Brownian motion implies that if $(\boldsymbol{W}_t:t\geq 0)$ is planar Brownian motion starting at $\boldsymbol{0}$, then for any fixed $\rho>0$, the processes $\big(\frac{1}{\rho}\boldsymbol{W}_{\rho^2 t}:t\geq 0\big)$ and $(\boldsymbol{W}_t:t\geq 0)$ have the same distribution. In particular, for any $\rho>0$, the joint distribution of the first hitting time and place of the horizontal ray $\mathcal{R}_\rho$ defined by \eqref{eq:ray}, \eqref{eq:first_time}, and \eqref{eq:first_place} satisfies the distributional equality 
\begin{equation}\label{eq:rho_scaling}
 (T_\rho, X_\rho) \stackrel{d}{=} \left(\rho^2 T_1, \rho X_1\right).
\end{equation}
This allows us to focus our attention on deriving the joint law of $(T_1, X_1)$.

The space-time transformations between planar Brownian motion starting at $\boldsymbol{0}$ and planar Brownian bridge starting and ending at $\boldsymbol{0}$, which are given by
\begin{align}
(\boldsymbol{W}_t:t\geq 0)&\stackrel{d}{=}\mleft((1+t)\boldsymbol{B}_{\frac{t}{1+t}}:t\geq 0\mright),\nonumber \\
(\boldsymbol{B}_t:0\leq t<1)&\stackrel{d}{=}\mleft((1-t)\boldsymbol{W}_{\frac{t}{1-t}}:0\leq t<1\mright),\label{eq:space_time}
\end{align}
will come in handy when proving Theorem \ref{thm:radial_distance}. They can be easily deduced from their well-known one-dimensional counterparts; see Exercise 3.10 in Chapter I of \cite{Revuz_Yor}.


\subsection{Conformal invariance and time change}\label{sec:conformal_invariance}

Recall the definition of the horizontal ray $\mathcal{R}_\rho$ from \eqref{eq:ray}. With the distributional identity \eqref{eq:rho_scaling} in mind, we now turn our attention to deriving the joint distribution of the first \emph{exit} time and place of the slit plane $\mathbb{R}^2\setminus\mathcal{R}_1$. This will yield the distribution of $(T_1, X_1)$ while allowing us to exploit the conformal invariance of planar Brownian motion. The particular result that we use can be found in \cite[Theorem 7.20]{Morters_Peres} or \cite[Theorem 2.2]{Lawler}, which we state here for the convenience of the reader in a form tailored to our purposes. While the theorem is stated in terms of Brownian motion on the complex plane $\mathbb{C}$, it is straightforward to adapt the result to our setting of $\mathbb{R}^2$. Indeed, we often identify $\mathbb{R}^2$ with $\mathbb{C}$ in the obvious way when it is convenient. 

First we recall some definitions and notation before stating the result. A \emph{domain} is a nonempty connected open subset of $\mathbb{C}$, and a \emph{conformal map} is a mapping between domains that is both analytic and bijective. For any $z\in\mathbb{C}$, we denote by $\mathbb{P}_z$ the law under which the process $(Z_t:t\geq 0)$ is Brownian motion in the complex plane starting at $z$. Furthermore, the first exit time of a domain $\mathcal{U}\subset\mathbb{C}$ is defined by $\tau_\mathcal{U}:=\inf\{t\geq 0:Z_t\notin \mathcal{U}\}$, with the convention $\inf\varnothing=\infty$.

\begin{theorem*}
Let $\mathcal{U},\mathcal{V} \subset \mathbb{C}$ be domains, and let $f:\overline{\mathcal{U}}\to \overline{\mathcal{V}}$ be continuous and map $\mathcal{U}$ conformally onto $\mathcal{V}$. Then for any $z\in\mathcal{V}$, we have the distributional equality
\begin{equation}\label{eq:conformal_invariance}
\left(\tau_\mathcal{V},Z_{\tau_\mathcal{V}}\right)\text{under}~\mathbb{P}_z\stackrel{d}{=}\left(\int_0^{\tau_\mathcal{U}} \big|f'(Z_t)\big|^2\D{t},\,f(Z_{\tau_\mathcal{U}})\right)\text{under}~\mathbb{P}_{f^{-1}(z)}.
\end{equation}
\end{theorem*}

Our strategy will be to use the equality in distribution \eqref{eq:conformal_invariance} to express the sought-after first exit time and place of $\mathcal{V}=\mathbb{R}^2\setminus\mathcal{R}_1$ in terms of planar Brownian motion run in a simpler domain $\mathcal{U}$ whose geometry allows us to fully exploit the independence of the coordinates of $(\boldsymbol{W}_t:t\geq 0)$, namely, the upper half-plane $\mathcal{H}:=\{(x,y)\in \mathbb{R}^2:y>0\}$. The reader can check that, when written as a function of a complex variable $z$, a suitable choice for the conformal map $f:\mathcal{H}\to \mathbb{R}^2\setminus\mathcal{R}_1$ is given by $f(z)=z^2+1$. In particular, $|f'(x+iy)|^2=4(x^2+y^2)$, and the preimage of the origin in $\mathbb{R}^2\setminus\mathcal{R}_1$ is the point $(0,1)$ in $\mathcal{H}$. Now we can link the distribution of $(T_1,X_1)$ with planar Brownian motion starting at the point $(0,1)\in \mathcal{H}$ by
\begin{equation}\label{eq:half_plane}
(T_1,X_1)\stackrel{d}{=}\left(4\int_0^{\tau_\mathcal{H}}\big(W_t^{(1)}\big)^2+\big(W_t^{(2)}\big)^2\D{t},\,\big(W_{\tau_\mathcal{H}}^{(1)}\big)^2+1\right)\text{under}~\mathbb{P}_{(0,1)}.
\end{equation}

At this juncture, the astute reader is sure to notice that under $\mathbb{P}_{(0,1)}$, the random variable $\tau_\mathcal{H}$ appearing on the right-hand side of \eqref{eq:half_plane} is simply the first passage time to the level $1$ of one-dimensional Brownian motion starting at $0$, and that conditioning on $\tau_\mathcal{H}$ makes $W_{\tau_\mathcal{H}}^{(1)}$ a normal random variable with mean $0$ and variance $\tau_\mathcal{H}$. Furthermore, after conditioning on both $\tau_\mathcal{H}$ and $W_{\tau_\mathcal{H}}^{(1)}$, the processes $\big((W_t^{(1)})^2:0\leq t\leq \tau_\mathcal{H}\big)$ and $\big((W_t^{(2)})^2:0\leq t\leq \tau_\mathcal{H}\big)$ become independent squared Bessel bridges of dimension $1$ and $3$, respectively. These observations take us to the next section, where we recall some explicit formulas for Laplace transforms of quadratic functionals of Bessel bridges.


\subsection{Quadratic functionals of Bessel bridges}\label{sec:quad_functionals}

Here we recall the rudiments of \emph{Bessel}$(\delta)$ \emph{bridges} of dimension $\delta=1$ and $\delta=3$ that will be needed to state the Laplace transform identities that were alluded to in the paragraph immediately following \eqref{eq:half_plane}. Loosely speaking, a Bessel$(1)$ bridge from $x\geq 0$ to $y\geq 0$ of duration $t>0$ is simply the absolute value of one-dimensional Brownian motion starting at $x$ and conditioned to be at $y$ at time $t$. In other words, a Bessel$(1)$ bridge is just the absolute value of an ordinary Brownian bridge with the same parameters. On the other hand, a Bessel$(3)$ bridge from $x\geq 0$ to $y\geq 0$ of duration $t>0$ is a Brownian bridge from $x$ to $y$ of duration $t$ whose path is conditioned to remain nonnegative over the entire time interval $[0,t]$. A Bessel$(3)$ bridge from $0$ to $0$ of duration $1$ is also known as a \emph{normalized Brownian excursion}. 

Brownian motion being at a specific point at a specific time is a null event, as is the bridge remaining nonnegative when either the start or endpoint is $0$, so both of these conditionings require some care in order to make our informal definitions precise; see Chapter VI, Section 3 and Chapter XI, Sections 1 and 3 in \cite{Revuz_Yor} for more details. It is worth pointing out that a Bessel$(3)$ bridge will almost surely be \emph{strictly positive} except possibly at the start or endpoint. This should be contrasted with the fact that a Bessel$(1)$ bridge hits $0$ with positive probability regardless of the start or endpoint. Indeed, if a Bessel$(1)$ bridge hits $0$ at all, then it does so uncountably many times. 

For $\delta>0$ and $x,y\geq 0$, squaring the coordinate process of a Bessel$(\delta)$ bridge from $\sqrt{x}$ to $\sqrt{y}$ of duration $t>0$ results in a \emph{squared Bessel}$(\delta)$ \emph{bridge} from $x$ to $y$ of duration $t>0$. To make this statement more precise, let $\mathbb{P}_{x,y}^{\delta,t}$ and $\mathbb{Q}_{x,y}^{\delta,t}$ denote, respectively, the laws under which the coordinate process $(X_s:0\leq s\leq t)$ is a Bessel$(\delta)$ and squared Bessel$(\delta)$ bridge from $x$ to $y$ of duration $t$. Then we have
\[
\mleft(X_s:0\leq s\leq t\mright)~\text{under}~\mathbb{Q}_{x,y}^{\delta,t}\stackrel{d}{=}\mleft(X^2_s:0\leq s\leq t\mright)~\text{under}~\mathbb{P}_{\sqrt{x},\sqrt{y}}^{\delta,t}.
\]

With these definitions in place, we can now state the formula from which our needed Laplace transform identities will be derived. This result is taken from Corollary 3.3 in Chapter XI of \cite{Revuz_Yor}.

\begin{theorem*}
For any $x,y\geq 0$ and $\delta>0$, let $\mathbb{Q}_{x,y}^{\delta,1}$ denote the law and expectation operator under which the coordinate process $(X_s:0 \leq s \leq 1)$ is a squared Bessel$(\delta)$ bridge from $x$ to $y$ of duration $1$. Then for any $b\geq 0$, we have 
\begin{equation}\label{eq:RY_identity}
\begin{split}
\mathbb{Q}_{x,0}^{\delta,1}\left[\exp\left(-\frac{b^2}{2} \int_0^1 X_s\D{s}\right)\right]&=\mathbb{Q}_{0,x}^{\delta,1}\left[\exp\left(-\frac{b^2}{2} \int_0^1 X_s\D{s}\right)\right]\\
&=\left(\frac{b}{\sinh b}\right)^\frac{\delta}{2}\exp\left(\frac{x}{2}(1-b\coth b)\right).
\end{split}
\end{equation}
\end{theorem*}

The identity \eqref{eq:RY_identity} needs to be generalized so that it applies to bridges of arbitrary duration $t>0$. Fortunately, Bessel bridges inherit the scaling property of Brownian motion in the sense that a bridge of duration $1$ can be scaled longer or shorter to produce a bridge of any desired duration with an appropriately scaled start and endpoint. This property leads to the following distributional identities for both types of bridges which hold for any $\delta, t>0$ and $x,y\geq 0$
\begin{align}
\left(X_s:0\leq s\leq t\right)\text{under}~\mathbb{P}_{x,y}^{\delta,t}&\stackrel{d}{=}\left(\sqrt{t}X_{s/t}:0\leq s\leq t\right)\text{under}~\mathbb{P}_{\frac{x}{\sqrt{t}},\frac{y}{\sqrt{t}}}^{\delta,1},\nonumber \\
\left(X_s:0\leq s\leq t\right)\text{under}~\mathbb{Q}_{x,y}^{\delta,t}&\stackrel{d}{=}\left(t\,X_{s/t}:0\leq s\leq t\right)\text{under}~\mathbb{Q}_{\frac{x}{t},\frac{y}{t}}^{\delta,1}.\label{eq:bridge_scaling}
\end{align}

In particular, for any $b\geq 0$, we can use \eqref{eq:bridge_scaling}, the change of variables $s\mapsto tu$, and \eqref{eq:RY_identity} to write
\begin{align*}
\mathbb{Q}_{x,0}^{\delta,t}\left[\exp\left(-\frac{b^2}{2} \int_0^t X_s\D{s}\right)\right]
&=\mathbb{Q}_{\frac{x}{t},0}^{\delta,1}\left[\exp\left(-\frac{b^2}{2}\int_0^t t\,X_{s/t}\D{s}\right)\right]\\
&=\mathbb{Q}_{\frac{x}{t},0}^{\delta,1}\left[\exp\left(-\frac{(bt)^2}{2}\int_0^1 X_u\D{u}\right)\right]\\
&=\left(\frac{bt}{\sinh(bt)}\right)^\frac{\delta}{2}\exp\left(\frac{x}{2t}\big(1-bt\coth(bt)\big)\right).
\end{align*}
Note that the same calculation can be done for the bridge from $0$ to $x$. Hence, for any $\delta,t>0$ and $b,x\geq 0$, we are led to the desired Laplace transform identity
\begin{equation}\label{eq:bridge_formula}
\begin{split}
&\mathbb{Q}_{x,0}^{\delta,t}\left[\exp\left(-\frac{b^2}{2} \int_0^t X_s\D{s}\right)\right]=\mathbb{Q}_{0,x}^{\delta,t}\left[\exp\left(-\frac{b^2}{2} \int_0^t X_s\D{s}\right)\right]\\
&=\left(\frac{bt}{\sinh(bt)}\right)^\frac{\delta}{2}\exp\left(\frac{x}{2t}\big(1-bt\coth(bt)\big)\right).
\end{split}
\end{equation}


\subsection{Complementary error function results}\label{sec:erfc}

The complementary error function can be interpreted as the right-tail of a half-normal random variable with variance $\frac{1}{2}$. More precisely, if $Z$ denotes a standard normal random variable, then we have
\begin{align}
\erfc(x)=\frac{2}{\sqrt{\pi}}\int_x^\infty e^{-u^2}\D{u}&=2\frac{1}{\sqrt{2\pi\frac{1}{2}}}\int_x^\infty e^{-\frac{u^2}{2\frac{1}{2}}}\D{u}\nonumber \\
&=\mathbb{P}\mleft(\mleft|\frac{Z}{\sqrt{2}}\mright|\geq x\mright).\label{eq:erfc_Z}
\end{align}
This representation can be used to compute the negative fractional moments of a nonnegative random variable whose Laplace transform features the complementary error function in a particular way. This is done by establishing a connection to the well-known moments of the half-normal distribution 
\begin{equation}\label{eq:half_normal}
\mathbb{E}\big[|Z|^p\big]=\sqrt{\frac{2^p}{\pi}}\Gamma\mleft(\frac{p}{2}+\frac{1}{2}\mright),~~p>-1;
\end{equation}
see, for instance, Equation (3.9) of \cite{Janson}.

\begin{lemma}\label{lem:erfc_moments}
Let $c>0$ and $0<\alpha\leq \frac{1}{2}$ be constants. Then there exists a nonnegative random variable $X$ whose Laplace transform is given by
\begin{equation}\label{eq:erfc_Laplace}
\mathbb{E}\mleft[e^{-s X}\mright]=\erfc(c\, s^\alpha),~~s\geq 0.
\end{equation}
In Janson's terminology \cite{Janson}, $X$ has ``moments of gamma type'' given by the formula
\begin{equation}\label{eq:erfc_moments}
\mathbb{E}\mleft[X^{-p}\mright]=\frac{\Gamma\mleft(\frac{p}{2\alpha}+\frac{1}{2}\mright)}{c^{p/\alpha}\sqrt{\pi}\,\Gamma(p+1)},~~ p>-\alpha.
\end{equation}
\end{lemma}

\begin{proof}
To see why such a random variable $X$ exists, first note that $s\mapsto\erfc(\sqrt{s})$ is \emph{completely monotone} and that $s\mapsto c^2 s^{2\alpha}$ is a \emph{Bernstein function}; see Chapters 1 and 3 of \cite{Bernstein} for the definitions of these properties. Moreover, by Theorem 3.7 of the same reference, we have that a completely monotone function composed with a Bernstein function is itself completely monotone. Now it follows from Bernstein's theorem \cite[Theorem 1.4]{Bernstein} that $s\mapsto\erfc(c\, s^\alpha)$ is the Laplace transform of a positive measure on $[0,\infty)$ with total mass equal to $\erfc(0)=1$.

Now that we have established the existence of the random variable $X$, the proof of the moment formula boils down to two applications of Tonelli's theorem along with analytic continuation. We start by proving the formula \eqref{eq:erfc_moments} for $p>0$. With this assumption, the first application of Tonelli's theorem results in
\begin{equation}\label{eq:Laplace_trick}
\int_0^\infty s^{p-1}\mathbb{E}\mleft[e^{-s X}\mright]\D{s}=\mathbb{E}\mleft[\int_0^\infty s^{p-1}e^{-s X}\D{s}\mright]=\Gamma(p)\,\mathbb{E}\mleft[X^{-p}\mright].
\end{equation}
After rearranging \eqref{eq:Laplace_trick}, we can use the hypothesis \eqref{eq:erfc_Laplace}, the half-normal representation \eqref{eq:erfc_Z}, and the fact that $s\mapsto c\, s^\alpha$ is strictly increasing to arrive at
\begin{align}
\mathbb{E}\mleft[X^{-p}\mright]&=\frac{1}{\Gamma(p)}\int_0^\infty s^{p-1}\erfc(c\, s^\alpha)\D{s}\nonumber \\
&=\frac{1}{p\,\Gamma(p)}\int_0^\infty p\,s^{p-1}\mathbb{P}\mleft(\mleft|\frac{Z}{c\sqrt{2}}\mright|^{1/\alpha}\geq s\mright)\D{s}.\label{eq:tail_integral1}
\end{align}
Now the integral appearing in \eqref{eq:tail_integral1} can be rewritten using Tonelli's theorem as
\begin{align}
\mathbb{E}\mleft[\int_0^\infty p\,s^{p-1}\mathbbm{1}_{\mleft\{\mleft|\frac{Z}{c\sqrt{2}}\mright|^{1/\alpha}\geq s\mright\}}\D{s}\mright]&=\mathbb{E}\mleft[\int_0^{\mleft|\frac{Z}{c\sqrt{2}}\mright|^{1/\alpha}} p\,s^{p-1}\D{s}\mright]\nonumber \\
&=\mathbb{E}\mleft[\mleft|\frac{Z}{c\sqrt{2}}\mright|^{p/\alpha}\mright].\label{eq:tail_integral2}
\end{align}
Substituting \eqref{eq:tail_integral2} into \eqref{eq:tail_integral1} while using \eqref{eq:half_normal} and the recurrence for $\Gamma$ leads to
\begin{align}
\mathbb{E}\mleft[X^{-p}\mright]&=\frac{1}{p\,\Gamma(p)}\frac{1}{c^{p/\alpha}\sqrt{2^{p/\alpha}}}\sqrt{\frac{2^{p/\alpha}}{\pi}}\Gamma\mleft(\frac{p}{2\alpha}+\frac{1}{2}\mright)\nonumber \\
&=\frac{\Gamma\mleft(\frac{p}{2\alpha}+\frac{1}{2}\mright)}{c^{p/\alpha}\sqrt{\pi}\,\Gamma(p+1)},~~ p>0.\label{eq:positive_moments}
\end{align}

Lastly, we extend the validity of the moment formula \eqref{eq:positive_moments} from $p>0$ to $p>-\alpha$. This can be done using an analytic continuation argument since the right-hand side of \eqref{eq:positive_moments} as a function of $p\in\mathbb{C}$ is analytic in the right half-plane $\{p\in\mathbb{C}:\Re p>-\alpha\}$; see Theorem 2.1 of \cite{Janson} for a precise statement and proof.
\end{proof}


\section{Proofs of the main results}

\subsection{Proof of Theorem \ref{thm:joint_Laplace}}\label{sec:joint_Laplace}

Here we prove Theorem \ref{thm:joint_Laplace} by combining the results noted in Subsections \ref{sec:scaling}, \ref{sec:conformal_invariance}, and \ref{sec:quad_functionals}. We start by using the distributional identity \eqref{eq:half_plane} to write
\begin{align}
&\mathbb{E}\left[\exp\left(-\frac{b^2}{8} T_1 - \mu X_1\right)\right]\nonumber \\
&=\mathbb{E}\left[\exp\left(-\frac{b^2}{2}\int_0^\tau \big(W_t^{(1)}\big)^2+\big(W_t^{(2)}\big)^2\D{t}-\mu\left(\big(W_\tau^{(1)}\big)^2+1\right)\right)\right],\label{eq:Laplace_expectation}
\end{align}
where $b,\mu\geq 0$ are constants, $(W_t^{(i)}:t\geq 0)$, $i=1,2$, are independent one-dimensional Brownian motions with $W_0^{(1)}=0$ and $W_0^{(2)}=1$, and $\tau$ is the first passage time to $0$ of the second Brownian motion. It is well known that conditioning on $\tau$ makes the process $\big((W_t^{(2)})^2:0\leq t\leq \tau\big)$ a squared Bessel$(3)$ bridge from $1$ to $0$ of duration $\tau$; see the introduction of \cite{first_passage}. Moreover, since the two Brownian motions are independent, $W_\tau^{(1)}$ is a normal random variable with mean $0$ and variance $\tau$ when conditioned on $\tau$. If we further condition on the value of $W_\tau^{(1)}$, then $\big((W_t^{(1)})^2:0\leq t\leq \tau\big)$ becomes a squared Bessel$(1)$ bridge from $0$ to $(W_\tau^{(1)})^2$, also of duration $\tau$. 

The conditional relationship between $W_\tau^{(1)}$ and $\tau$ that we just described, when coupled with the well-known density for $\tau$, leads to the joint density
\begin{align}
\mathbb{P}\left(W_\tau^{(1)}\in \D{x},\, \tau\in \D{t}\right)&=\frac{1}{\sqrt{2\pi t}}e^{-\frac{x^2}{2t}}\D{x}\frac{1}{\sqrt{2\pi t^3}}e^{-\frac{1}{2t}}\D{t}\nonumber \\
&=\frac{1}{2\pi t^2}e^{-\frac{1+x^2}{2t}}\D{x}\D{t},~~x\in\mathbb{R},\,t>0.\label{eq:mixing_density}
\end{align}
We will use this joint density to disintegrate \`{a} la \cite{bridges} the path measure under which the expectation on the right-hand side of \eqref{eq:Laplace_expectation} is taken, thereby expressing it as a mixture of integral functionals of squared Bessel bridges. 

Towards this end, we introduce the following abbreviated notation for the Laplace transforms from Section \ref{sec:quad_functionals}, where $b,x\geq 0$ and $t>0$, namely
\begin{align*}
\Psi_1(b;x,t)&:=\mathbb{Q}_{0,x}^{1,t}\left[\exp\left(-\frac{b^2}{2}\int_0^t X_s\D{s}\right)\right],\\
\Psi_3(b;t)&:=\mathbb{Q}_{1,0}^{3,t}\left[\exp\left(-\frac{b^2}{2} \int_0^t X_s\D{s}\right)\right].
\end{align*}
Now we can use \eqref{eq:Laplace_expectation}, \eqref{eq:mixing_density}, and the aforementioned disintegration to write
\begin{align}
&\mathbb{E}\left[\exp\left(-\frac{b^2}{8} T_1 - \mu X_1\right)\right]\nonumber \\
&=\int_0^\infty \int_{-\infty}^\infty \frac{1}{2\pi t^2}e^{-\frac{1+x^2}{2t}}\Psi_1(b;x^2,t)\,\Psi_3(b;t)\,e^{-\mu(x^2+1)}\D{x}\D{t}\nonumber \\
&=\int_0^\infty \frac{1}{2\pi t^2}e^{-\frac{1}{2t}}\Psi_3(b;t)\int_{-\infty}^\infty \Psi_1(b;x^2,t)\,e^{-\mu(x^2+1)-\frac{x^2}{2t}}\D{x}\D{t}.\label{eq:disintegration}
\end{align}

After replacing $\Psi_1(b;x^2,t)$ with the formula from \eqref{eq:bridge_formula}, we see that the inner integral in \eqref{eq:disintegration} is just a Gaussian integral, so we can compute it explicitly as
\begin{align}
&\int_{-\infty}^\infty \Psi_1(b;x^2,t)\,e^{-\mu(x^2+1)-\frac{x^2}{2t}}\D{x}\nonumber \\
&=\int_{-\infty}^\infty \sqrt{\frac{bt}{\sinh(bt)}}\exp\left(\frac{x^2}{2t}\big(1-bt\coth(bt)\big)\right)\,e^{-\mu(x^2+1)-\frac{x^2}{2t}}\D{x}\nonumber \\
&=e^{-\mu}\sqrt{\frac{bt}{\sinh(bt)}}\int_{-\infty}^\infty \exp\left(-\frac{x^2}{2}\big(b\coth( bt)+2\mu\big)\right)\D{x}\nonumber \\
&=e^{-\mu}\sqrt{\frac{bt}{\sinh (bt)}}\sqrt{\frac{2\pi}{b\coth(bt)+2\mu}}.\label{eq:inner_integral}
\end{align}

Now we can replace the $\Psi_3(b;t)$ and the inner integral appearing in \eqref{eq:disintegration} with the formulas from \eqref{eq:bridge_formula} and \eqref{eq:inner_integral}, respectively, to get
\begin{align}
&\mathbb{E}\left[\exp\left(-\frac{b^2}{8} T_1 - \mu X_1\right)\right]\nonumber \\
&=\int_0^\infty \frac{1}{2\pi t^2}e^{-\frac{1}{2t}}\Psi_3(b;t)\,e^{-\mu}\sqrt{\frac{bt}{\sinh (bt)}}\sqrt{\frac{2\pi}{b\coth(bt)+2\mu}}\D{t}\nonumber \\
&=\frac{e^{-\mu}}{2\pi}\int_0^\infty e^{-\frac{b}{2}\coth(bt)}\big(b\csch(bt)\big)^2\sqrt{\frac{2\pi}{b\coth(b t)+2\mu}}\D{t}.\label{eq:coth_integral}
\end{align}

We can evaluate the definite integral in \eqref{eq:coth_integral} by making a sequence of two substitutions, the first of which is $u=b\coth(bt)$. With this substitution, we have $\D{u}=-(b\csch(bt))^2\D{t}$, and the old limits of integration $0$ and $\infty$ become the new limits $\infty$ and $b$, respectively. This results in the integral identity
\[
\mathbb{E}\left[\exp\left(-\frac{b^2}{8} T_1 - \mu X_1\right)\right]=\frac{e^{-\mu}}{2\pi}\int_b^\infty e^{-\frac{1}{2}u}\sqrt{\frac{2\pi}{u+2\mu}}\D{u}.
\]
Next, we make the substitution $u=2v^2-2\mu$. Hence, $\D{u}=4v\D{v}$, and the new limits of integration are $\sqrt{b/2+\mu}$ to $\infty$. This allows us to write
\begin{align}
\mathbb{E}\left[\exp\left(-\frac{b^2}{8} T_1 - \mu X_1\right)\right]&=\frac{e^{-\mu}}{2\pi}\int_{\sqrt{b/2+\mu}}^\infty e^{-v^2+\mu}\sqrt{\frac{2\pi}{2v^2}}4v\D{v} \nonumber \\
&=\frac{2}{\sqrt{\pi}}\int_{\sqrt{b/2+\mu}}^\infty e^{-v^2}\D{v}\nonumber \\
&=\erfc\left(\sqrt{\frac{b}{2}+\mu}\right).\label{eq:erfc_integral}
\end{align}

The last step in the proof of Theorem \ref{thm:joint_Laplace} involves using the distributional equality \eqref{eq:rho_scaling} to generalize the Laplace transform identity \eqref{eq:erfc_integral} to any $\rho>0$, as well as writing the identity in terms of $\lambda\geq 0$ instead of $b\geq 0$. In particular, we must have $b^2/8=\lambda\rho^2$, so $b=2\rho\sqrt{2\lambda}$ and the identity becomes
\begin{align*}
\mathbb{E}[\exp(-\lambda T_\rho - \mu X_\rho)]&=\mathbb{E}[\exp(-\lambda\rho^2 T_1 - \mu\rho X_1)]\\
&=\erfc\left(\sqrt{\rho\left(\sqrt{2\lambda}+\mu\right)}\right).
\end{align*}
\hfill\qed

\subsection{Proof of Theorem \ref{thm:conditional_structure}}\label{sec:conditional_structure}

Here we prove the remarkable conditional relationship between $T_\rho$ and $X_\rho$ that is claimed by Theorem \ref{thm:conditional_structure}. First note that setting $\lambda=0$ in Theorem \ref{thm:joint_Laplace} yields the marginal Laplace transform $\mathbb{E}[\exp(- \mu X_\rho)]=\erfc(\sqrt{\rho\mu})$. We can use this together with the full statement of Theorem \ref{thm:joint_Laplace} to deduce the identity
\[
\mathbb{E}\left[\exp\left(-\big(\sqrt{2\lambda}+\mu\big)X_\rho\right)\right]=\erfc\left(\sqrt{\rho\left(\sqrt{2\lambda}+\mu\right)}\right)=\mathbb{E}[\exp(-\lambda T_\rho - \mu X_\rho)],
\]
which holds for all $\lambda,\mu\geq 0$ and $\rho>0$. Rewriting the right-hand side using the conditional expectation of the bounded random variable $e^{-\lambda T_\rho}$ given $X_\rho$ results in
\begin{equation}\label{eq:marginal_joint}
\mathbb{E}\left[e^{-\mu X_\rho}e^{-\sqrt{2\lambda}X_\rho}\right]=\mathbb{E}\big[e^{-\mu X_\rho}\mathbb{E}[e^{-\lambda T_\rho}|X_\rho]\big].
\end{equation}

Next, fix $\lambda\geq 0$ and $\rho>0$, and note that the random variable $\mathbb{E}[e^{-\lambda T_\rho}|X_\rho]$ is measurable with respect to the $\sigma$-algebra generated by $X_\rho$. Hence, there exists a Borel function $g_\lambda:[\rho,\infty)\to [0,1]$ such that $\mathbb{E}[e^{-\lambda T_\rho}|X_\rho]=g_\lambda(X_\rho)$ almost surely; see \cite[Lemma 1.13]{Kallenberg}. Now for any $\mu\geq 0$, we can use \eqref{eq:marginal_joint} to write 
\begin{equation}\label{eq:positive_measures}
\mathbb{E}\left[e^{-\mu X_\rho}e^{-\sqrt{2\lambda}X_\rho}\right]=\mathbb{E}\big[e^{-\mu X_\rho}g_\lambda(X_\rho)\big].
\end{equation}

Both sides of \eqref{eq:positive_measures} can be interpreted as the Laplace transforms of finite measures which are absolutely continuous with respect to the law of $X_\rho$. More precisely, define the finite measures $\nu_1$ and $\nu_2$ supported on $[\rho,\infty)$ by
\[
\nu_1(\D{x}):=e^{-\sqrt{2\lambda}\, x}\,\mathbb{P}(X_\rho\in\D{x})~~\text{ and }~~\nu_2(\D{x}):=g_\lambda(x)\,\mathbb{P}(X_\rho\in\D{x}).
\]
Then \eqref{eq:positive_measures} implies that for all $\mu\geq 0$, we have
\[
\int_{[0,\infty)} e^{-\mu x}\nu_1(\D{x})=\int_{[0,\infty)} e^{-\mu x}\nu_2(\D{x}).
\]
Consequently, $\nu_1=\nu_2$ by the uniqueness theorem for Laplace transforms of finite measures on $[0,\infty)$. In particular, it follows that for all $\lambda\geq 0$, 
\[
e^{-\sqrt{2\lambda}X_\rho}= \mathbb{E}[e^{-\lambda T_\rho}|X_\rho]~~\text{almost surely}.
\]

The proof is complete once we recognize $e^{-\sqrt{2\lambda}\,x}$ with $x>0$ as the well-known Laplace transform of the first passage time to the level $x$ of one-dimensional Brownian motion starting at $0$; see Proposition 3.7 in Chapter II of \cite{Revuz_Yor}.
\hfill\qed


\subsection{Proof of Theorem \ref{thm:densities}}\label{sec:densities}

In this section we compute the joint and marginal densities of $(T_\rho,X_\rho)$. We start by computing the marginal density \eqref{eq:X_density} of $X_\rho$, which we then combine with the conditional structure of Theorem \ref{thm:conditional_structure} to derive the other densities. Towards this end, note that from \eqref{eq:rho_scaling} and \eqref{eq:half_plane} we get the equality in distribution 
\begin{equation}\label{eq:X_marginal}
X_\rho\stackrel{d}{=}\rho \left(\big(W_{\tau_\mathcal{H}}^{(1)}\big)^2+1\right)\text{under}~\mathbb{P}_{(0,1)},
\end{equation}
where $(W_t^{(1)}:t\geq 0)$ is the horizontal coordinate of planar Brownian motion starting at the point $(0,1)$ and running until it first exits the upper half-plane $\mathcal{H}$ at the random time $\tau_\mathcal{H}$. In this scenario, it is well known that $W_{\tau_\mathcal{H}}^{(1)}$ has the distribution of a standard Cauchy random variable $C$; see Proposition 3.11 in Chapter III of \cite{Revuz_Yor}. Hence, for any $x\geq \rho$ we can use \eqref{eq:X_marginal} and symmetry to write
\begin{align*}
\mathbb{P}(X_\rho\leq x)&=\mathbb{P}\mleft(\rho(C^2+1)\leq x\mright)\\
&=2\,\mathbb{P}\mleft(0\leq C\leq\sqrt{x/\rho-1}\mright)\\
&=\frac{2}{\pi}\arctan\mleft(\sqrt{x/\rho-1}\mright).
\end{align*}
Differentiating with respect to $x>\rho$ leads to the desired marginal density \eqref{eq:X_density}.

Next, we use the marginal density \eqref{eq:X_density} of $X_\rho$ along with the conditional relationship between $X_\rho$ and $T_\rho$ from Theorem \ref{thm:conditional_structure} to compute their joint density. This requires the well-known density of the first passage time to the level $x>0$ of one-dimensional Brownian motion starting at $0$; see Section 3 of Chapter III in \cite{Revuz_Yor}. Putting everything together results in the joint density \eqref{eq:joint_density}, that is,
\begin{align*}
\mathbb{P}(T_\rho\in\D{t},\,X_\rho\in \D{x}) &=\frac{x}{\sqrt{2\pi t^3}}e^{-\frac{x^2}{2t}}\D{t}\frac{\sqrt{\rho}}{\pi x\sqrt{x-\rho}}\D{x}\\
&= \frac{\sqrt{\rho}\,e^{-\frac{x^2}{2t}}}{\sqrt{2\pi^3 t^3 (x-\rho) }}\D{t}\D{x}, ~~ t>0,\, x>\rho.
\end{align*}

Lastly, we deduce the marginal density \eqref{eq:T_density} for $T_\rho$ by integrating the $x$ variable in the newfound joint density \eqref{eq:joint_density} from $\rho$ to $\infty$. We do this by first making the substitution $x=u^2+\rho$. Hence, $\D{x}=2u\D{u}$, and the new limits of integration are now $0$ to $\infty$. This allows us to write
\begin{align}
\mathbb{P}(T_\rho\in\D{t})&=\frac{\sqrt{\rho}}{\sqrt{2\pi^3 t^3}}\D{t}\int_0^\infty \frac{e^{-\frac{(u^2+\rho)^2}{2t}}}{u}2u\D{u}\nonumber \\
&=\frac{\sqrt{2\rho}\,e^{-\frac{\rho^2}{2t}}}{\sqrt{\pi^3 t^3}}\D{t}\int_0^\infty \exp\left(-\frac{1}{2t}u^4-\frac{\rho}{t}u^2\right)\D{u},~~t>0.\label{eq:K_integral}
\end{align}
To evaluate the integral on the right-hand side of \eqref{eq:K_integral}, we can use the identity
\begin{equation}\label{eq:GR_identity}
\int_0^\infty e^{-a x^4-2b x^2}\D{x}=\frac{1}{4}\sqrt{\frac{2b}{a}}\exp\mleft(\frac{b^2}{2a}\mright)K_\frac{1}{4}\mleft(\frac{b^2}{2a}\mright),~~\Re a\geq 0, \,b>0;
\end{equation}
see \cite[Equation 3.469.1]{Gradshteyn_Ryzhik} or \cite[Section 7]{GR_project} for several proofs. We recover the marginal density \eqref{eq:T_density} from \eqref{eq:K_integral} by taking $a=\frac{1}{2t}$ and $b=\frac{\rho}{2t}$ in \eqref{eq:GR_identity} to get
\begin{align*}
\mathbb{P}(T_\rho\in\D{t})&=\frac{\sqrt{2\rho}\,e^{-\frac{\rho^2}{2t}}}{\sqrt{\pi^3 t^3}}\D{t}\frac{\sqrt{2\rho}}{4}\exp\mleft(\frac{\rho^2}{4t}\mright)K_\frac{1}{4}\mleft(\frac{\rho^2}{4t}\mright)\\
&=\frac{\rho\,e^{-\frac{\rho^2}{4t}}}{2\sqrt{\pi^3 t^3}}K_\frac{1}{4}\mleft(\frac{\rho^2}{4t}\mright)\D{t},~~t>0.
\end{align*}
\hfill\qed


\subsection{Proof of Theorem \ref{thm:radial_distance}}\label{sec:radial_distance}

In this section we compute the densities and moments of the radial distances $R^{\mathrm{BM}}$ and $R^{\mathrm{BB}}$ which are defined in \eqref{eq:radial_distance}. We start by deducing the distributional equality $R^{\mathrm{BM}}\stackrel{d}{=} T_1^{-\frac{1}{2}}$, which follows from using the scaling relation \eqref{eq:rho_scaling} to write
\begin{align}
\mathbb{P}\mleft(R^\mathrm{BM}<\rho\mright)=\mathbb{P}\mleft(T_\rho>1\mright)&=\mathbb{P}\mleft(T_1>\frac{1}{\rho^2}\mright)\label{eq:radial_time}\\
&=\mathbb{P}\mleft(T_1^{-\frac{1}{2}}<\rho\mright).\nonumber
\end{align}
Now it is straightforward to derive the density of $R^\mathrm{BM}$ from \eqref{eq:radial_time} and \eqref{eq:T_density}. 

To compute the moments of $R^\mathrm{BM}$, we first obtain from Theorem \ref{thm:joint_Laplace} the marginal Laplace transform $\mathbb{E}[\exp(-\lambda T_1)]=\erfc\big((2\lambda)^{1/4}\big)$. Using Lemma \ref{lem:erfc_moments} with $c=2^{1/4}$ and $\alpha=\frac{1}{4}$ along with the  distributional equality yields 
\[
\mathbb{E}\mleft[\mleft(R^{\mathrm{BM}}\mright)^p\mright]=\mathbb{E}\mleft[T_1^{-\frac{p}{2}}\mright]=\frac{\Gamma\mleft(p+\frac{1}{2}\mright)}{\sqrt{2^p \pi}\,\Gamma\mleft(\frac{p}{2}+1\mright)},~~p>-\frac{1}{2}.
\]

The bridge case is treated next. While the end result is simpler, this case requires a bit more work. The basic idea is to use the space-time transformation \eqref{eq:space_time} to express the probability that planar Brownian bridge never hits the ray $\mathcal{R}_\rho$ as the probability that planar Brownian motion with negative horizontal drift never hits the same ray. We can then compute this latter probability using Girsanov's theorem and the joint Laplace transform from Theorem \ref{thm:joint_Laplace}.

We start by noticing that because $\boldsymbol{B}_0=\boldsymbol{B}_1$, we have $R^\mathrm{BB}<\rho$ if and only if $\boldsymbol{B}_t \notin \mathcal{R}_\rho$ for all $0\leq t<1$. Now the space-time transformation \eqref{eq:space_time} can be used along with the change of variables $s=\frac{t}{1-t}$ to write
\begin{align}
\mleft\{R^\mathrm{BB}<\rho\mright\}&=\mleft\{(1-t)\boldsymbol{W}_\frac{t}{1-t} \notin \mathcal{R}_\rho~\text{for all}~0\leq t<1\mright\}\nonumber \\
&=\mleft\{\frac{1}{1+s}\boldsymbol{W}_s \notin \mathcal{R}_\rho~\text{for all}~s\geq 0\mright\}\nonumber \\
&=\mleft\{\boldsymbol{W}_s \notin \mathcal{R}_{\rho+\rho s}~\text{for all}~s\geq 0\mright\}\nonumber \\
&=\big\{\boldsymbol{W}_s -(\rho s,0)\notin \mathcal{R}_\rho~\text{for all}~s\geq 0\big\}.\label{eq:bridge_drift}
\end{align}

Define the vector $\boldsymbol{h}_\rho:=(\rho,0)$. Then the process $(\boldsymbol{W}_s -(\rho s,0):s\geq 0)$ appearing in the event on the right-hand side of \eqref{eq:bridge_drift} is simply planar Brownian motion with the constant drift vector $-\boldsymbol{h}_\rho$. In particular, letting $(M_t:t\geq 0)$ denote the exponential martingale defined by 
\[
M_t=\exp\mleft(-\boldsymbol{W}_t\cdot\boldsymbol{h}_\rho-\frac{1}{2}|\boldsymbol{h}_\rho|^2 t\mright),
\]
we can use the Girsanov and optional stopping theorems to show that 
\begin{align}
\mathbb{P}\big(\boldsymbol{W}_s -\boldsymbol{h}_\rho s\in \mathcal{R}_\rho~\text{for some}~0\leq s\leq n\big)&=\mathbb{E}[\mathbbm{1}_{\{T_\rho\leq n\}}M_n]\nonumber \\
&=\mathbb{E}\big[\mathbb{E}[\mathbbm{1}_{\{T_\rho\leq n\}}M_n|\mathcal{F}_{T_\rho\wedge n}]\big]\nonumber \\
&=\mathbb{E}\mleft[\mathbbm{1}_{\{T_\rho\leq n\}}M_{T_\rho\wedge n}\mright]\nonumber \\
&=\mathbb{E}\mleft[\mathbbm{1}_{\{T_\rho\leq n\}}M_{T_\rho}\mright].\label{eq:bounded_Girsanov}
\end{align}
In light of \eqref{eq:bridge_drift}, letting $n\to\infty$ on both sides of \eqref{eq:bounded_Girsanov} while using monotone convergence leads to 
\begin{align}
\mathbb{P}(R^\mathrm{BB}<\rho)&=1-\mathbb{P}\big(\boldsymbol{W}_s -\boldsymbol{h}_\rho s\in \mathcal{R}_\rho~\text{for some}~s\geq 0\big)\nonumber \\
&=1-\mathbb{E}\mleft[\mathbbm{1}_{\{T_\rho<\infty\}}\exp\mleft(-\boldsymbol{W}_{T_\rho}\cdot\boldsymbol{h}_\rho-\frac{1}{2}|\boldsymbol{h}_\rho|^2 T_\rho\mright)\mright]\nonumber \\
&=1-\mathbb{E}\mleft[\mathbbm{1}_{\{T_\rho<\infty\}}\exp\mleft(-\rho\,X_\rho-\frac{1}{2}\rho^2 T_\rho\mright)\mright].\label{eq:Girsanov}
\end{align}

Since planar Brownian motion (without drift) almost surely hits $\mathcal{R}_\rho$ in finite time, the indicator appearing on the right-hand side of \eqref{eq:Girsanov} is redundant. Hence, we can use the joint Laplace transform formula from Theorem \ref{thm:joint_Laplace} along with the representation \eqref{eq:erfc_Z} involving the standard normal $Z$ to rewrite \eqref{eq:Girsanov} as
\[
\mathbb{P}(R^\mathrm{BB}<\rho)=1-\erfc\mleft(\sqrt{2}\,\rho\mright)=\mathbb{P}\mleft(\frac{1}{2}|Z|<\rho\mright).
\]
Now it follows that $R^{\mathrm{BB}}\stackrel{d}{=} \frac{1}{2}|Z|$ and that $R^{\mathrm{BB}}$ has the half-normal distribution with variance $\frac{1}{4}$. Moreover, we can get from \eqref{eq:half_normal} the moment formula 
\[
\mathbb{E}\mleft[\mleft(R^{\mathrm{BB}}\mright)^p\mright]=\frac{1}{2^p}\mathbb{E}\big[|Z|^p\big]=\frac{\Gamma\mleft(\frac{p}{2}+\frac{1}{2}\mright)}{\sqrt{2^p \pi}},~~p>-1.
\]
\hfill\qed


\subsection{Proofs of Theorems \ref{thm:bridge_areas} and \ref{thm:BM_star_area}}\label{sec:areas}

First we prove the expected area formulas in the planar Brownian bridge case.
\begin{proof}[Proof of Theorem \ref{thm:bridge_areas}]
The planar Brownian bridge $(\boldsymbol{B}_t:0\leq t\leq 1)$ satisfies all three conditions of Lemma \ref{lem:mean_area_formula}, so we can use that result and Theorem \ref{thm:radial_distance} to write
\begin{align*}
\mathbb{E}\big[\mathrm{area}(\mathcal{S}^{\mathrm{BB}})\big]=\pi\,\mathbb{E}\mleft[\mleft(R^{\mathrm{BB}}\mright)^2\mright]&=\pi\frac{\Gamma\mleft(\frac{3}{2}\mright)}{\sqrt{4\pi}}\\
&=\frac{\pi}{4}.
\end{align*}
This proves Theorem \ref{thm:bridge_areas} since the other two identities are already known.
\end{proof}

We conclude by treating the expected areas in the nonbridge case.
\begin{proof}[Proof of Theorem \ref{thm:BM_star_area}]
The planar Brownian motion $(\boldsymbol{W}_t:0\leq t\leq 1)$ satisfies all three conditions of Lemma \ref{lem:mean_area_formula}, so we can use that result and Theorem \ref{thm:radial_distance} to write
\begin{align*}
\mathbb{E}\big[\mathrm{area}(\mathcal{S}^{\mathrm{BM}})\big]=\pi\,\mathbb{E}\mleft[\mleft(R^{\mathrm{BM}}\mright)^2\mright]&=\pi\frac{\Gamma\mleft(\frac{5}{2}\mright)}{\sqrt{4 \pi}\,\Gamma\mleft(2\mright)}\\
&=\frac{3\pi}{8}.
\end{align*}
This proves Theorem \ref{thm:BM_star_area} since the convex hull identity is already known and the topological hull inequality follows from the inclusion chain of Lemma \ref{lem:inclusion_chain}.
\end{proof}


\section{Future directions}\label{sec:open}

Here we discuss a few open questions that merit further study.

\begin{enumerate}

\item \textbf{Boundary of the star hull}\\
Unlike the convex hull, the boundary of the star hull inherits some of the roughness of the Brownian path so presumably it is nonrectifiable. However, it's not clear whether the boundary is truly fractal with Hausdorff dimension greater than $1$. On the other hand, the boundary of the topological hull is known to have Hausdorff dimension $\frac{4}{3}$. This was first conjectured by Mandelbrot \cite{Mandelbrot} and later proven by Lawler--Schramm--Werner \cite{frontier_dim}. See Figure \ref{fig:illustrations} for depictions of the boundaries of all three hulls.

\item \label{Q:top_area} \textbf{Expected area of $\mathcal{T}^{\mathrm{BM}}$}\\
As mentioned prior to Theorem \ref{thm:BM_star_area}, the expected area of the topological hull of planar Brownian motion is unknown, though the question has attracted some interest as evidenced by the MathOverflow posts \cite{Dominik, Nate}. The present author has produced a Monte Carlo estimate of $0.5911$ for this expected area by simulating $10^5$ random walk paths on $\mathbb{Z}^2$, each with $10^5$ steps, and then properly scaling and averaging the integer-valued areas computed by applying the shoelace formula to the outer boundary obtained from the algorithm described in \cite{Richard}. We note in passing that this estimate is close to another rational multiple of $\pi$, namely, $\frac{3\pi}{16}\approx 0.58905$.

\item \textbf{Second moments of the hull areas}\\
The second moment of the area of $\mathcal{T}^{\mathrm{BB}}$ was expressed as the integral of a complicated function in \cite{bubble_variance}, though the authors were unable to obtain an accurate numerical evaluation. No expression is known in the convex or star hull case for either process. However, it is worth mentioning that the second moment of the \emph{perimeter} of $\mathcal{C}^{\mathrm{BM}}$ was given as a complicated integral in \cite{Wade_Xu}; see \cite{hull_bounds} for remarks on the numerical evaluation of this integral.

\item \textbf{Symmetric $\alpha$-stable processes}\\
The closed convex hull of the symmetric $\alpha$-stable process in $\mathbb{R}^d$ has been studied in \cite{stable_hulls, Levy_hulls, stable_walk_hulls}, and there are formulas for its expected intrinsic volumes when $\alpha>1$; see \cite[Example 2.6]{Levy_hulls}. When $\alpha>1$ and $d=2$, the methods of the present paper can be used to compute the expected area of the star hull in this case as soon as the distribution of the corresponding ray-hitting time $T_1$ is known. See \cite{Isozaki_stable} for some partial results in this direction.

\end{enumerate}

\bigskip

\begin{acknowledgments}
The author would like to thank Chris Burdzy and Greg Lawler for helpful comments and discussions
concerning Theorem \ref{thm:conditional_structure}, as well as Zakhar Kabluchko and Peter Pivovarov for offering pointers to the literature on zonoid hulls and star hulls, respectively.
\end{acknowledgments}


\bibliography{star_bib}
\bibliographystyle{amsplainabbrev}

\end{document}